\documentclass[12pt,a4paper]{article}

\usepackage[a4paper,portrait]{geometry}
\usepackage{amsmath,amsfonts,latexsym,amssymb,amsthm}
\usepackage{graphicx}
\usepackage{epsfig}

\newcommand{\DOIFPDF}[2]{\ifx\pdfoutput\undefined #2\else#1\fi}

\newcommand{\EM}{\ensuremath}

\newcommand{\PDFABLE}[2]{%
\newif\ifpdf%
\ifx\pdfoutput\undefined\pdffalse%
\else\pdftrue\pdfoutput=1\pdfcompresslevel=9\fi%
\ifpdf%
 \usepackage#1
 \usepackage[pdftex,%
             a4paper,%
             colorlinks,%
             citecolor=blue,%
             pagebackref,%
             plainpages=false]{hyperref}%
\else%
 \usepackage#2
 \usepackage{url}
\fi}

\makeatletter
\newcommand{\@THMSTYLES}{%
  \newtheoremstyle{bodyrm}
  {3pt}
  {3pt}
  {}
  {}
  {\bfseries\sffamily}
  {.}
  { }
  {}
  \newtheoremstyle{bodyit}
  {3pt}
  {3pt}
  {\itshape}
  {}
  {\bfseries\sffamily}
  {.}
  { }
  {}
}
\newcommand{\THMEN}{%
  \@THMSTYLES
  \theoremstyle{bodyit}
  \newtheorem{thm}{Theorem}[section]%
  \newtheorem{cor}[thm]{Corollary}%
  \newtheorem{prop}[thm]{Proposition}%
  \newtheorem{lem}[thm]{Lemma}%
  \theoremstyle{bodyrm}%
  \newtheorem{defi}[thm]{Definition}%
  \newtheorem{xpl}[thm]{Example}%
  \newtheorem{exo}[thm]{Exercise}%
  \newtheorem{hyp}[thm]{Hypothesis}%
  \newtheorem{eur}[thm]{Heuristics}%
  \newtheorem{pro}[thm]{Problem}%
  \newtheorem{rem}[thm]{Remark}%
  \newtheorem{prp}[thm]{Property}%
}
\newcommand{\THMFR}{%
  \@THMSTYLES
  \theoremstyle{bodyit}
  \newtheorem{thm}{Théorème}[section]%
  \newtheorem{lem}[thm]{Lemma}%
  \theoremstyle{bodyrm}%
  \newtheorem{defi}[thm]{Définition}%
  \newtheorem{rem}[thm]{Remarque}%
}
\makeatother

\makeatletter
\newcommand{\SMALLSECS}{%
 \renewcommand{\section}{\@startsection%
  {section}
  {1}
  {0em}
  {\baselineskip}
  {0.5\baselineskip}
  {\normalfont\large\bfseries}}
 \renewcommand{\subsection}{\@startsection%
  {subsection}
  {2}
  {0em}
  {\baselineskip}
  {0.25\baselineskip}
  {\normalfont\bfseries}}
}
\makeatother

\makeatletter
\providecommand{\timenow}{\@tempcnta\time
\@tempcntb\@tempcnta
\divide\@tempcntb60
\ifnum10>\@tempcntb0\fi\number\@tempcntb
\multiply\@tempcntb60
\advance\@tempcnta-\@tempcntb
:\ifnum10>\@tempcnta0\fi\number\@tempcnta}
\makeatother

\makeatletter
\newcommand{\versiondetravail}{%
 \renewcommand{\@evenfoot}{%
 \hfil{\tiny\texttt{%
   Version préliminaire, compilée le \today{} à \timenow.}\hfill}}%
 \renewcommand{\@oddfoot}{\@evenfoot}%
}
\makeatother

\newcommand{\dR}{\EM{\mathbb{R}}}

\newcommand{\cA}{\EM{\mathcal{A}}}

\newcommand{\cC}{\EM{\mathcal{C}}}

\newcommand{\cN}{\EM{\mathcal{N}}}

\newcommand{\bE}{\EM{\mathbf{E}}}

\newcommand{\bP}{\EM{\mathbf{P}}}

\newcommand{\al}{\alpha}
\newcommand{\be}{\beta}

\newcommand{\ga}{\gamma}

\newcommand{\la}{\lambda}

\newcommand{\te}{\theta}
\newcommand{\ta}{\tau}

\newcommand{\veps}{\varepsilon}
\newcommand{\vphi}{\varphi}

\newcommand{\Det}[1]{\mathrm{Det}\,}

\newcommand{\WH}[1]{\widehat{#1}}

\renewcommand{\leq}{\leqslant}
\renewcommand{\geq}{\geqslant}

\SMALLSECS
\THMEN
\title{Estimation of the distribution of random shifts deformation}
\author{I.~\textsc{Castillo} \& J-M.~\textsc{Loubes}}
\date{}

\newcommand{\mykeywords}{
Semiparametric statistics,
Order two properties,
Penalized Maximum Likelihood,
Practical algorithms.
}

\newcommand{\mysubjclass}{
62G05, 62G20.
}

\newcommand{\uns}{\frac{1}{n}}

\newcommand{\rd}{\sqrt{2}}
\newcommand{\argm}{\operatorname{argmax}}

\newcommand{\Argm}[1]
{\underset{#1}{\argm\ }}
\newcommand{\mymax}[1]
{\underset{#1}{\ }}
\newcommand{\pli}{+\infty}

\begin{document}

\maketitle
{\small
\begin{abstract}
Consider discrete values of functions shifted by unobserved translation effects, which are independent
 realizations of a random variable with unknown distribution $\mu$, modeling the variability in
the response of each individual. Our aim is to construct a nonparametric estimator of the density
of these random translation deformations using semiparametric preliminary estimates of the shifts.
Building on results of Dalalyan et al. (2006), semiparametric estimators are obtained in our
 discrete framework and their performance studied. From these estimates we construct a
nonparametric estimator of the target density. Both rates of convergence and an algorithm to
construct the estimator are provided.\\
\end{abstract}
{ \noindent
 \textbf{Keywords}: \mykeywords \\
 \textbf{Subject Class. MSC-2000}: \mysubjclass}
}   
\section{Introduction}
\vspace{0.5cm}
Our aim is  to estimate the common density $\vphi$ of independent random variables $\theta_j, \: j=1,\dots,J_n$, with distribution $\mu$, observed in a panel data analysis framework in a translation model. More precisely, consider $J_n$ unknown curves $t\to f^{[j]}(t)$  sampled  at multiple points $t_{ij}=t_i=i/n,\: i=1,\dots,n$, with random i.i.d. translation effects $\te_j,\: j=1,\dots,J_n$, in the following regression framework
\begin{equation} \label{eq:pb}
 Y_{ij}= f^{[j]} (t_{ij}-\theta_j)+ \sigma \veps_{ij}~, \: i=1,\dots,n, \: j=1,\dots,J_n,
\end{equation}
where $\veps_{ij}$ are independent standard normal $\cN(0,1)$ random noise and are independent of the $\te_j$'s, while $\sigma$ is a positive real number which is assumed to be known.
The number of points per curve is denoted by $n$ while $J_n$ stands for
the number of curves.\vskip .1in
Equation \eqref{eq:pb} describes the situation  often encountered in biology, data mining or econometrics (see e.g \cite{Ronn} or \cite{Loubes04b})  where the outcome of an experiment depends on a random variable $\theta$ which models the case where the  data  variations take into account the variability of each individual:  each subject $j$ can react in a different way within a mean behaviour, with slight variations given by the unknown curves $f^{[j]}$. Estimating $\vphi$, the density of the unobserved $\theta_j$'s, enables to understand this mean behaviour.\vskip .1in
 Nonparametric estimation of $\vphi$ belongs to the class of inverse problems for which the subject of the inversion is a probability measure, since the  realizations $\theta_j$ are warped by  unknown functions $f^{[j]}$'s. Here, these functions  are unknown, hence the underlying inverse problem becomes more than harmful as  sharp approximations of the $\theta_j$'s are needed to prevent flawed rates of convergence for the density estimator.  While the estimation of  parameters, observed through their image by an operator, traditionally relies on the inversion of the operator, here the repetition of the observations enables to use recent advances in semiparametric estimation to improve the usual strategies developed to solve such a
  problem.\\
\indent Note that estimation of such warping parameters have been investigated by several authors using nonparametric methods for very general models,  see for instance \cite{Gasser95,Kneip92}, or  \cite{Wang99}. However little attention is paid to the law
     of these random parameters. Moreover, as said previously, sharp estimates of the parameters are required to achieve density estimation, which requires semiparametric methods.\vskip .1in
 Our approach consists, first, in the estimation of the shifts $\theta_j$ while the functions $f^{[j]}$ play the role of nuisance parameters. We follow the semiparametric approach introduced in \cite{Dalalyan03} in the Gaussian white noise framework and extend it to the discrete regression framework. This provides sharp estimators of the unobserved shifts, up to order 2 expansions. Alternative methods can be found in \cite{Loubes04} or \cite{Vimond}. These preliminary estimates enable, in a second time,  to recover the unknown density $\varphi$ of the
$\theta_j$'s  as if the shifts were directly observed, at least if $J_n$ is not significatively larger than $n$.
This paper also provides a practical algorithm, for both the semiparametric and the nonparametric steps. The first step is the most difficult one: to build practicable semiparametric estimators, we propose an algorithm which refines the one proposed in \cite{LLL04}
for the period model and relies on the previously obtained second order expansion.\\
\indent Beyond the shift estimation case, which involves a symmetry assumption on $f^{[j]}$, our procedure may be applied to semiparametric
models where an explicit penalized profile likelihood is available and well-behaved estimators of the $\theta_j$'s can be obtained.
 A particularly important example in applications is the
 estimation of the period of an unknown periodic function, see for instance \cite{LLL04}. Given a sequence of $J_n$ experiments like the one considered
in \cite{LLL04}, one might be interested in estimating the law of the corresponding periods of the signals.
 In this case one can also consider applying our method, under some conditions made explicit in the sequel.
 \vskip .1in The paper falls into the following parts. In Section \ref{siden}, semiparametric estimators $\widehat{\theta}_j$
of the realizations of the shift parameters  are proposed, and sharp bounds between $\widehat{\theta}_j$ and
$\theta_j$ are provided. Then, in Section \ref{npara}, a nonparametric estimator of the unknown distribution is considered while rates of convergence
 are provided in the case where $\mu$ admits a density, in the general model \eqref{eq:pb} under the
 condition that the $\te_j$'s can be sufficiently
 well approximated. In Section 4, the practical estimation problem is considered and
 a simulation study is conducted. Technical proofs are gathered in Section \ref{sappen}.

\section{Semiparametric Estimation of the shifts} \label{siden}
 In this Section, we provide, for each fixed $j$, semiparametric estimators of the $j^{\rm th}$ realization $\theta_j$
of the random variable $\theta$, observed in Model \eqref{eq:pb}. To build this estimates, we follow the method introduced by Dalalyan, Golubev and Tsybakov in \cite{Dalalyan03} for a continuous-time version of the translation model. We obtain analogues of two of their results in our discrete-time model: a deviation estimate stated in Lemma  \ref{tsy} and a second order expansion for the estimators stated in Lemma \ref{lem:tsy2}. We also establish a new result in Lemma \ref{espcond}, which enables to control the bias of the estimates.  The particular form of the estimators and the second order expansion will not be used to build the density estimator in Section \ref{npara}. \\

Hence, conditionally to the event $\theta_j=\theta$, we construct an estimator $\WH{\theta}_j$ and establish asymptotic results for the conditional distribution $\left(\WH{\theta_j}\ |\ \theta_j=\theta\right)$,  gathered in Lemmas \ref{tsy}, \ref{lem:tsy2} and \ref{espcond}. In the remaining of this Section, since $j$ is fixed, the index $j$  in the notation is dropped (for instance $Y_{ij}$ is simply denoted by $Y_i$).  We shall denote by $\|.\|$ the $L^2$-norm on $[0,1]$ and by $\|.\|_{\infty}$  the $L^{\infty}$-norm on $\dR$.

\subsection{Shift estimation in the discrete time translation model}
 The model reduces to, assuming for simplicity that $\sigma=1$,
\begin{equation} \label{mods}
 Y_{i}=f(t_{i}-\theta)+\veps_{i}\quad\quad i=1,\ldots,n,
\end{equation}
where $f$ is a {\em symmetric} function satisfying some additional assumptions detailed below
 and $t_i=i/n$.  The corresponding problem is the one of  semiparametric estimation of the center of symmetry
 in a discrete framework.

\begin{description}
\item[{Working assumptions in the translation model}]
\end{description}
We assume that the support $\Theta$ of the distribution $\mu$ of the random variable $\theta$ is compact and contained in an interval of diameter upper-bounded by $1/2$
\begin{align*}
(A1) & \quad  \Theta=\{\te,\ |\te|\leq\ta_{0}\}, \quad \text{where } \ta_{0}\ \ \text{is such that\ }
0<\ta_{0}<1/4.
\intertext{The function $f$ is assumed to be symmetric (that is, $f(x)=f(-x)$ for all real $x$) and periodic with period 1 with Fourier coefficients denoted by $f_k,\: k\geq 1$,}
 (A2) & \quad f(t) =\rd\sum_{k\geq 1}f_{k}\cos(2\pi kt),\quad\text{where}\
\ f_{k}=\rd\int_{0}^{1}f(t)\cos(2\pi kt)dt.
\intertext{Let $\mathcal{C}^2(\dR)$ denote the set of all twice continuously differentiable functions
on $\dR$. We assume that there exist $\rho>0$  and $C_0<\pli$ such
that $f$ belongs to the set $F$ defined by}
(A3) & \quad F =F(\rho,C_0)=\{f\in\mathcal{C}^2(\dR),\quad f_{1}^2\geq\rho,\quad \|f''\|^2\leq C_{0}\}.\\
\end{align*}
Conditions (A1)-(A3) can be seen as working assumptions and are essentially the same as in \cite{Dalalyan03}. Assuming periodicity of $f$ is not a drawback since, in practice, the function $f$ is compactly supported and can easily be periodicized. The assumption that $f$ belongs to $\cC^2(\dR)$ is handy in particular for proving the second order properties of the estimator. Note also that for simplicity in the definitions of the classes, as in \cite{Dalalyan03}  we have assumed that the Fourier coefficient
for $k=0$, that is $\int_0^1 f(u)du$, is zero.\\
\\
Identifiability in model (2) follows from : symmetry, 1-periodicity of the functions
(note that assuming that $f_1^2\geq\rho$ implies that $f$ cannot be periodic of
smaller period) and the fact that the diameter of $\Theta$ is less than $1/2$. \\
\\
Note also that within this framework, the Fisher information for estimating $\te_j$ for a fixed $j$ is, as $n$
tends to $\pli$, given by $\{1+o(1)\}n\|f'\|^2$.
\begin{description}
\item[{Construction of the estimator}]
\end{description}
Let us define an estimator $\WH{\te}$ of the shift $\theta$ in model \eqref{mods} by
\begin{equation} \label{definit}
\WH{\te}=\Argm{\tau\in\Theta}\ \sum_{k\geq 1}h_{k}\left(\uns \sum_{i=1}^{n}\cos(2\pi
k(t_{i}-\ta))Y_{i}\right)^2,
\end{equation}
 where $(h_{k})$ is a sequence of real numbers in $[0,1]$ satisfying some conditions
made precise in the following subsection. The sequence $(h_{k})$ is called {\em sequence of weights} or {\em filter}.

The estimator $\WH{\te}$ is similar to the estimator $\WH{\theta}_{PML}$ proposed in \cite{Dalalyan03}:
here the integral in their definition is replaced by the equivalent discrete sum in the
discrete-time model. 
As we sketch below, the estimator \eqref{definit} arises in a natural way by using a penalized profile likelihood method as in \cite{Dalalyan03}, though here in an approximate way only, 

First one turns the study of the regression model into the study of a sequence of independent submodels. Let us introduce, for any $k\geq 1$,
\begin{eqnarray*}
x_{k}=\uns\sum_{i=1}^{n}\rd\cos(2\pi k t_{i})Y_{i}, &
\quad \xi_k=\frac{1}{\sqrt{n}}\sum_{i=1}^{n}\rd\cos(2\pi k t_{i})\veps_{i}.\\
x_{k}^{*}=\uns \sum_{i=1}^{n}\rd\sin(2\pi k t_{i})Y_{i}, &
\quad \xi_k^*=\frac{1}{\sqrt{n}}\sum_{i=1}^{n}\rd\sin(2\pi k t_{i})\veps_{i}.\\
\end{eqnarray*}
Note that $x_k$ and $x_k^*$ are observed. Using the fact that $Y_i$ follows \eqref{mods},
\begin{equation}\label{proj}
x_{k}=\cos(2\pi k\te)f_{k}+d_{k,n}+\frac{1}{\sqrt{n}}\xi_{k},
\end{equation}
\begin{equation}\label{proj*}
x_{k}^{*}= \sin(2\pi k\te)f_{k}+d_{k,n}^{*}+\frac{1}{\sqrt{n}}\xi_{k}^{*},
\end{equation}
where $d_{k,n}, d_{k,n}^{*}$ 
 are terms of difference between the Fourier coefficient and its approximation:
$$d_{k,n}=\sqrt{2}\left(\uns\sum_{i=1}^{n}\cos(2\pi kt_{i})f(t_{i}-\te)-
\int_{0}^{1} \cos(2\pi k t)f(t-\te)dt \right).$$
The term $d_{k,n}^{*}$ is obtained in a
similar way replacing the cosine by a sine. We would like to underline two important facts about the previous quantities. First, since the $\veps_i$'s are Gaussian $\cN(0,1)$, the variables $(\xi_{k},\xi_{k}^{*})_{k\geq 1}$
are also Gaussian and, since we assume that $t_i=i/n$, using the orthogonality
of the trigonometric basis over this system of points, they are in fact independent standard Normal.
Second, it is important to note that both $d_{k,n}$ and $d_{k,n}^{*}$ are non-random and bounded uniformly in $\te$. We will use more precise bounds in function of $k$ and $n$ in the proofs, see Lemma \ref{lemapprox} in Section
\ref{sappen}.



The penalized profile likelihood method is as follows. For each integer $k$ and $\ta\in\Theta$, 
 let us define the quantity $ p_{\ta}(x_{k},x_{k}^{*},f_{k})$ as
 $$\left(\frac{1}{\sqrt{2\pi}}\right)^3\exp\left(-\frac{n}{2}(x_{k}-\cos(2\pi
 k\ta)f_{k}-d_{k,n})^2-\frac{n}{2}(x_{k}^{*}-\sin(2\pi
 k\ta)f_{k}-d_{k,n}^{*})^2-\frac{f_{k}^2}{2\sigma_{k}^2}\right).$$ This is the usual likelihood corresponding to
 the observation $(x_k,x_k^*)$ with an additional penalization term $-f_k^2/2\sigma_k^2$, where $\sigma_k$ has to
 be chosen. The profile likelihood technique (see \cite[Chap. 25]{vdv}), consists in "profiling out" the
 nuisance parameter $f_k$ by setting
 \begin{align*}
 f_k^*(\ta)&=\Argm{f_k}p_{\ta}(x_{k},x_{k}^{*},f_{k})\\
 \WH{\te}_{PML}&=\Argm{\ta\in\Theta}\prod_{k\geq 1} p_{\ta}(x_{k},x_{k}^{*},f_{k}^*(\ta)).
 \end{align*}
Here a difficulty is that $d_{k,n}$ and $d_{k,n}^*$ depend on $f_k$. However, if we neglect those terms, we can follow the calculations made in \cite{Dalalyan03} and we obtain that $\WH{\te}_{PML}$ is the maximizer of
$$\sum_{k\geq 1}h_{k}\left(\uns \sum_{i=1}^{n}\cos(2\pi k\ta)x_k + \sin(2\pi k\ta)x_k^*\right)^2,$$
which is exactly the same as \eqref{definit} if we set $h_k=\sigma_{k}^2/(\sigma_{k}^2+n^{-1})$. Thus the criterion \eqref{definit} can be obtained by an {\em approximate} profile likelihood method. Yet, it is not trivial to see at this point if having neglected the terms $d_{k,n}, d_{k,n}^*$ in the construction of the estimator can have a negative influence over the behavior of the criterion \eqref{definit}. In fact, we will see in the sequel that this is not the case, and that \eqref{definit} can lead to a very good estimator, even at second order, provided a sensible choice of $(h_k)$ is made.

\subsection{Asymptotic behavior of the shifts estimators} \label{sestim}
First let us precise some technical assumptions we make on the sequence of weights $(h_k)$ in the definition
 \eqref{definit} of the estimator. These are the same as Assumptions B and C in \cite{Dalalyan03}, except that here
we also restrict ourselves to a finite number of nonzero weights.
\begin{description}
\item[Assumptions on the sequence of weights $(h_k)$]
\end{description}
Let the sequence $(h_{k})$ be such that $h_{1}=1$, $0\leq h_{k}\leq 1$ for all
$k\geq 1$ and assume that there are positive constants $D_{1}$ and $\rho_{1}$ such that
\begin{align*}
(C1)\quad& \text{The number of weights such that}\ h_{k}\neq 0\ \text{is finite}. \\
(C2)\quad&  \left[\sum_{k\geq 1} (2 \pi k)^2 h_k^2\right]^{1/2} \geq\rho_{1}(\log^2{n})\max_{k\geq 1}(2\pi k)h_{k}.\\
(C3)\quad& \sum_{k\geq 1}h_{k}(2\pi k)^4\leq D_{1}n.\\
(T)\quad & \left(\sum_{k\geq 1}(1-h_{k})(2\pi k)^2f_{k}^2\right)^2=
o\left(\sum_{k\geq 1}(1-h_{k})^2(2\pi k)^2 f_{k}^2\right).
\end{align*}
The first condition is quite natural to make the estimator feasible in practice, Conditions (C2) and (C3) precise the range of the sequence $(h_{k})$. Condition (T) allows to obtain second order properties, see the proof of
Lemma~\ref{lem:tsy2}.\\
\begin{rem} As noted in \cite{Dalalyan03}, conditions (C1), (C2), (C3) and (T) are fulfilled for a quite
wide range of weights. For instance, the sequences $(h_{k}={\bf 1}_{1\leq k\leq N(T)})$, also called projection
weights, satisfy the preceding conditions since (C2) and (C3) are satisfied respectively for $N(T)\geq
C\log^4{n}$ and $N(T)\leq Cn^{1/5}$, while condition (T) is always satisfied for projection weights since
$\sum_{k\geq N(T)}(2\pi k)^2f_{k}^2\rightarrow 0$, as $n\rightarrow +\infty$, due to (A3).\\
\end{rem}
\begin{rem}
It is also be possible to consider random, data-driven, weights. This approach is considered in \cite{Dalalyan07}.
\end{rem}
\begin{description}
\item[Asymptotic properties]
\end{description}
For easiness of reference in Section 3, it is convenient here to make the dependence in $j$ explicit again. In the following Lemmas, $f^{[j]'}$ and $f^{[j]}_k$ respectively denote the derivative and the Fourier coefficients of $f^{[j]}$. Lemmas \ref{tsy} and \ref{lem:tsy2} are respectively analogues of Lemma 5 and
Theorem 1 in \cite{Dalalyan03} here in a discrete regression framework, which seems to be closer to practical applications.  Though it seems natural that the cited results extend to our context, it is not obvious that the extra terms induced by the discretization of the model, for instance the $d_{n,k}$'s introduced above, do not
interfere with the rates, in particular at the second order. But we prove that they eventually do not,
see the proof of the Lemmas in Section \ref{sappen}.
\begin{lem}[Deviation bound]\label{tsy}
Assume that $\operatorname{(A), (C), (T)}$ are fulfilled. For any $K>0$ and any positive integer $n$, denote $x_{n}=K\sqrt{\log{n}}$. There exist positive constants  $c_1,\: c_2$ such that for any $K>0$, for $n$ large enough,  uniformly in $j\in\{1,\dots,J_n\}$, $\theta_j \in
\Theta$ and $f^{[j]}\in F$, it holds
\begin{equation} \label{optbound} \bP \left( \sqrt{n}|\WH{\theta}_{j} - \theta_j| > x_{n} |\theta_j \right)
 \leq c_1 \exp(-c_2 x_{n}^2).
\end{equation}
\end{lem}
\vskip.1in
\begin{lem}[Second order Expansion] \label{lem:tsy2} Assume that $\operatorname{(A), (C), (T)}$ are fulfilled. Let us denote  $R^n[h,f^{[j]}]=\sum_{k=1}^\infty (2\pi k)^2[(1-h_k)^2f^{[j]\,2}_k+h_k^2/n]$.
 Uniformly in $j\in\{1,\dots,J_n\}$, $\theta_j \in
\Theta$ and $f^{[j]}\in F$, as $n$ tends to $+\infty$,
 \begin{equation} {\bf E}\left((\WH{\theta}_{j}-\theta_j)^2|\theta_j\right) = \frac{1}{n \|f^{[j]'}\|^2} \left( 1+(1+o(1))
 \frac{R^n[h,f^{[j]}]}{\|f^{[j]'}\|^2} \right). \label{o2}\\
 \end{equation}
\end{lem}
Lemma \ref{lem:tsy2} implies that, conditionally to $\theta_j$, the estimator $\WH{\theta}_j$ is efficient for estimating $\theta_j$ at the first order. It also provides an explicit form
 for the second order term of the quadratic risk.
 The explicit expression of this term is not needed  to  establish the convergence rate of the plug-in estimator in Section \ref{npara}. Nevertheless it justifies the choice of the filter made in Section \ref{ssimul}. Indeed, we see from \eqref{o2} that an appropriate filter $(h_k)$ is a
  filter such that $R^n[h,f^{[j]}]$ is as small as possible. \\

The following Lemma \ref{espcond} is new with respect to \cite{Dalalyan03}. It ensures that the conditional
  law of $\WH{\te}_j$ is centered at $\te_j$, up to a $O(\log{n}/n)$ term.\\

\begin{lem}[Asymptotical Bias]\label{espcond}
Assume $\operatorname{(A), (C), (T)}$ and that there exists a positive constant $D$ such that for any $f\in F$, it holds $\sum_{k\geq 1} k^2 |f_k| \leq D$.  Then, uniformly in $j\in\{1,\dots,J_n\}$, as $n$ tends to $+\infty$,
\begin{equation} \label{eq:espcond}
\bE\left((\WH{\te}_j-\te_j) | \te_j \right) = O\left(\frac{\log{n}}{n}\right).\\
\end{equation}
\end{lem}
This Lemma requires only slightly more regularity on $f$ than a second derivative bounded in $L^2$, which is what condition (A3) imposes, that is $\sum_{k\geq 1}k^4 f_k^2 \leq C_0$. 
It enables us to have slightly broader framework for our results in Section \ref{npara}. However, one can still obtain interesting results without using this Lemma, see Remark \ref{remcond} after Theorem \ref{nonparam}.

\subsection{Case of the period model}

Let us now consider the period model mentioned in the introduction, where symmetry of the functions is not assumed.
The random variables $\theta_j$ arise this time as period of periodic functions. The observations in a fixed and equally spaced design are
\begin{equation}\label{eqperiod}
 Y_{ij}=f^{[j]}\left(\frac{i}{n\theta_j}\right)+\veps_{ij}\quad\quad i=-n/2,\ldots,n/2~,j=1,\ldots, J_n~,
\end{equation}
where the $\theta_j$'s belong to a compact interval in $]0,\pli[$  and the 1-periodic functions $f^{[j]}$ fulfill some smoothness assumptions, for instance the ones assumed in \cite{IC05} in the Gaussian white noise framework.

It is established in \cite{IC05} that
the penalized profile likelihood method yields estimators satisfying, with appropriate rescaling,
statements similar to \eqref{optbound} and \eqref{o2}, in the  continuous-time model, 
see \cite[Lemma 11 and Theorem 1]{IC05}.

Hence one can apply the method of this paper for estimating
the law of the $\theta_j$'s in model \eqref{eqperiod}, provided one can transpose the proofs of \eqref{optbound}-\eqref{o2} for the continuous-time model in terms of the discrete framework, as is done here in Section 5 for the shift model.



\section{Nonparametric estimation of the distribution $\mu$} \label{npara}

We are interested in the estimation of the distribution of the unobserved sample
$\theta_1,\dots,\theta_{J_n}$ in model \eqref{eq:pb}.
We shall assume that the number of curves
$J_n$ tends to $\pli$. Our approach is based on the assumption that, along each curve, one can estimate in an
appropriate way the corresponding $\theta_j$.

More precisely, the realizations $\theta_1,\dots,\theta_{J_n}$ are unknown but we assume that they can be
approximated by some preliminary estimators $\WH{\theta}_{j,n}$, for $j=1,\dots,J_n$, denoted for simplicity  $\WH{\theta}_{j}$ in the sequel. 
\begin{defi} \label{defiapp}
We say that the random variables $\WH{\theta}_{1},\ldots,\WH{\theta}_{J_n}$ {\em approximate} the sample $\te_1,\ldots,\te_{J_n}$ if for each $j$, the variable $\WH{\theta}_{j}$ is built using the observations $Y_{1j},\ldots,Y_{nj}$ (i.e is measurable with respect to these observations)  and satisfies the deviation bound \eqref{optbound} given in Lemma~\ref{tsy}. 
\end{defi} 
Note in particular that with this definition, the random variables $\WH{\theta}_j$ are independent. 
The fact that \eqref{optbound} holds roughly means that the $\theta_j$'s are approximated 
by the $\WH{\theta}_{j}$'s at almost parametric rate with an exponential control of the deviation probability. 

In Section \ref{siden}, we have studied in details a possible way of obtaining $\WH{\theta}_j$'s satisfying this approximation property in model \eqref{eq:pb}. However, we would like to point out that the results of Section 
\ref{npara} hold as long as the $\WH{\theta}_j$'s are approximations of the $\theta_j$'s in model \eqref{eq:pb} in the sense of Definition \ref{defiapp} (for Theorem \ref{nonparam} below, we shall also assume that \eqref{eq:espcond} holds), which possibly allows using estimators produced by other methods. Extensions to frameworks beyond model \eqref{eq:pb} could also be considered.

\subsection{A discrete estimator of $\mu$.}

A first way to define an estimator of $\mu$ is to consider a plug-in version of the usual empirical
distribution,  defined using the preliminary estimates $\WH{\theta}_{j}$ as
\begin{equation}
  \label{eq:1}
  \WH{\mu}_{J_n}=\frac{1}{J_n} \sum_{j=1}^{J_n} \delta_{\WH{\theta}_{j}}.
\end{equation}
The empirical distribution computed with the conditional estimators  of the shifts provides a consistent
approximation of the distribution of the true random shifts, in a weak sense.\\

\begin{thm}[Weak consistency of the plugged empirical measure]  \label{wcle}
Assume  \eqref{optbound}, that $\Theta$ is compact  and that there are positive finite constants $\alpha$ and $B$ such that $J_n \leq B n^{\alpha}$ and $J_n \rightarrow + \infty$. Then it holds
  \begin{equation}
    \label{eq:2}
    \WH{\mu}_{J_n} \stackrel{J_n \rightarrow \infty}{\rightharpoonup} \mu \quad \text{almost surely},
  \end{equation}
which means that for all continuously differentiable compactly supported function $g$,
$$ \WH{\mu}_{J_n} g=\frac{1}{J_n} \sum_{j=1}^{J_n} g(\WH{\theta}_j) \rightarrow \mu g =\bE(g(\theta))
\quad \text{almost surely}.\\ $$
\end{thm}

\begin{proof}
For $g$ a continuously differentiable compactly supported function, we get that
\begin{align*}
   \WH{\mu}_{J_n} g & = \frac{1}{J_n} \sum_{j=1}^{J_n} \left( g(\WH{\theta}_j)-g(\theta_j) \right) \quad (I) \\
 & + \frac{1}{J_n} \sum_{j=1}^{J_n} g(\theta_j) \quad \quad \quad \quad \quad (II).
\end{align*}
The law of large numbers ensures that, almost surely,
\begin{equation} (II) \stackrel{J_n \rightarrow \infty}{\longrightarrow} \bE(g(\theta)). \label{conclu1}
\end{equation}
Now Taylor upper bound leads to $
  | \frac{1}{J_n} \sum_{j=1}^{J_n} \left( g(\WH{\theta}_j)-g(\theta_j) \right)|  \leq  \frac{1}{J_n}
  \sum_{j=1}^{J_n} \|g^{'}\|_\infty |\WH{\theta}_j - \theta_j | .$
If $ \| g^{'} \|_{\infty} =0 $, then the previous quantity is equal to zero. Now consider the case
$\|g^{'}\|_{\infty} \neq 0$. Hence, using prior bound and \eqref{optbound}, we get for any $\la \geq 0$
\begin{align*}
  \bP \left(  | \frac{1}{J_n} \sum_{j=1}^{J_n} [ g(\WH{\theta}_j)-g(\theta_j) ] | \geq \la \: \: |\: \: \theta_1,\dots,\theta_{J_n} \right) &
  \leq \sum_{j=1}^{J_n}
  \bP \left( |\WH{\theta}_j - \theta_j| \geq \frac{\la}{\|g^{'}\|_\infty}\: \:  |\: \: \theta_j \right) \\
& \leq c_1 J_n \exp\left(-c_2 \frac{\la^2 n \|f^{[j]\,'}\|^2}{\|g^{'}\|_\infty^2}\right),
\end{align*}
which is uniform in $(\theta_j)_{j=1,\dots,J_n}$. Then, choosing $\la = c \sqrt{ \log n / n}$  leads to the
following bound.
$$ \bP \left(  | \frac{1}{J_n} \sum_{j=1}^{J_n} [ g(\WH{\theta}_j)-g(\theta_j) ] | \geq \la \right) \leq
c_1 J_n n^{-\alpha} n^{\alpha -c_2 c^2  \|f^{[j]\,'}\|^2/\|g^{'}\|_\infty^2}.$$ For $c$ large enough, namely for all
$\eta \geq 0$, $c^2 \geq (1+\alpha+\eta) \|g^{'}\|_\infty^2 / (c_2 \|f^{[j]\,'}\|^2)$, we can write
$$  \bP \left(  | \frac{1}{J_n} \sum_{j=1}^{J_n} [ g(\WH{\theta}_j)-g(\theta_j) ] | \geq c \sqrt{\frac{\log n}{n}} \right)
 \leq c_1 n^{-(1+\eta)}.$$
Borel Cantelli's Lemma enables us to conclude that a.s.
\begin{equation}
  \label{conclu2}
  \frac{1}{J_n} \sum_{j=1}^{J_n} [ g(\WH{\theta}_j)-g(\theta_j) ] \stackrel{J_n\rightarrow + \infty} \longrightarrow 0.
\end{equation}
Finally \eqref{conclu1} and \eqref{conclu2} prove the result.
\end{proof}
 Hence, we have constructed a discrete estimator of the law of the random shifts. Nevertheless,
 in many cases this estimator is too rough when the law of the unknown effect has a density, said $\vphi$, with respect
 to Lebesgue's measure. That is the reason why, in the following, a density estimator is built,
for which we provide functional rates of convergence.

\subsection{Estimation of the density of the random deformation}
Consider the following kernel estimator of $\varphi$, the density of the unobserved $\theta_j$'s in model \eqref{eq:pb}, based on a kernel $K$, to be specified in the following,
 and on the quantities $\WH{\theta}_j$. For all $x$ in $\Theta$, let us define
\begin{equation}\label{defnp}
\hat{\vphi}(x) = \frac{1}{J_n h_n}\sum_{j=1}^{J_n} K \left(\frac{x-\WH{\theta}_j}{h_n} \right).
\end{equation}
In this subsection, we shall assume that the quantities $\WH{\te_j}$ satisfy the approximation property stated in Definition \ref{defiapp} and also the 
control on their expectation provided by \eqref{eq:espcond}. We have checked in Section 2 that both properties
are fulfilled under some regularity conditions for the estimators $\WH{\te_j}$ built in Section 2.1.

Let us denote by $H_M(\be,L)$ the set of all densities $\varphi$ with support included in the interval $[-\ta_0,\ta_0]=\Theta$, which belong to the H\"older class $H(\be,L)$ (see \cite{tsybook}, p.5) and are uniformly bounded by a positive constant $M$.  

For clarity in the following statement, we shall
assume that for $n$ large enough, either $J_n \leq (n/\log n)^\frac{2\beta+1}{\beta+2}$ or the converse
inequality hold, for $\beta$ defined below (otherwise use a subsequence argument).
\vspace{.3cm}
\begin{thm}[Rate of convergence of the nonparametric estimator]
  \label{nonparam}
Let us assume that $\varphi$ belongs to the class $H_M(\be,L)$ for some positive $L$ and $M$ and 
with $\be>1$. Assume moreover \eqref{optbound}, \eqref{eq:espcond} and that the kernel $K$ is smooth, compactly
supported, of order $\lfloor\be\rfloor$.
Then the kernel estimator $\hat{\vphi}$ defined by \eqref{defnp} achieves the following rates of convergence, as $n$ and $J_n$ tend to $+\infty$,
\begin{equation}\label{ratecv}
\sup_{x\in\Theta}\sup_{\varphi\in H_M(\be,L)}\bE\left(\left[\hat{\varphi}(x)-\varphi(x)\right]^2\right) = \begin{cases}
O\left(J_{n}^{-\frac{2\be}{2\be+1}}\right), & \: {\rm if} \:  J_n \leq (n/\log n)^\frac{2\beta+1}{\beta+2} \\
O\left((n/ \log n)^{-\frac{2\be}{\be+2}}\right), & \: {\rm if} \: J_n \geq  (n/\log n)^\frac{2\beta+1}{\beta+2}
\end{cases}
\end{equation}
\end{thm}
 Thus the classical rate of convergence $J_n^{-\frac{2\be}{2\beta+1}}$ of density estimators over H\"older classes
 $H(\be,L)$, with $\beta>1$, is maintained, provided the number of curves $J_n$ does not exceed  $(n/\log n)^\frac{2\beta+1}{\beta+2}$. In fact, it can be checked, using standard lower bound techniques, that in model \eqref{eq:pb}, the minimax rate of convergence of the pointwise mean-squared risk for the estimation of $\vphi$ over the considered H\"older-class is not faster than constant times $J_n^{-\frac{2\be}{2\beta+1}}$, which yields the rate-optimality of the procedure in this model in the first case of the Theorem. In the other case, the number of points per curve $n$ becomes the limiting factor, and a slower rate specified by \eqref{ratecv} is obtained. Whether this second rate is optimal is an open question.

Our results can be interpreted as follows. The inverse problem is drastically reduced when the number of observations per subject increases,  enabling, in a way, to invert the convolution operator.
  A nonparametric estimation of the density of the unobserved parameter in a regression
framework can only be achieved if there are numerous observations for each curve. In our case, the fact that
asymptotics can be taken both in $J_n$ and in $n$ enables us to estimate first, for each curve,
the random effect and then plug the values to estimate the density.  If the number of observations  per curve is small, as it is usually the case in pharmacokinetics, such techniques cannot be applied  and we refer to \cite{Loubes04b} for an alternative methodology.

\begin{proof}
For simplicity in the notation, we assume throughout the proof that the 
 $\WH{\theta}_j$'s are identically distributed - let us recall that they are independent -, which enables us to just deal with $j=1$. If this is not the case, then one can still use the independence and then bound the different quantities arising $j$ by $j$. Also we denote $h_n$ simply by $h$. 

First note that the bias-variance decomposition is, for any $x$ in $\Theta$,
\begin{eqnarray*}
\bE\left([\hat{\varphi}(x)-\varphi(x)]^2\right)  =  & \left(\bE [\hat{\varphi}(x)]-\varphi(x)\right)^2 &+
\quad \bE\left([\hat{\varphi}(x)-\bE(\hat{\varphi}(x))]^2\right)\\
= & b(x)^2 & +\quad v(x).
\end{eqnarray*}
Let us denote by $\Delta$ the quantity $\hat{\te}_1-\te_1$. Note that, by definition of $\hat{\te}_1$, $\Delta$
is a measurable function of $(\te_1,\{\veps_{i1}\}_{i=1,\ldots,n})$.
 In the sequel, we denote $\Delta=g(\te_1,\veps)$.\\
\\
Let us denote by $\cA_1=\{|\WH{\te_1}-\te_1|\leq D(n^{-1}\log{n})^{1/2}\}$. Using \eqref{optbound},
the probability of its complement is negligible.\\
\\
A Taylor expansion of the kernel $K$ at the order $k$ yields the existence of a random variable
$Z$ such that
\begin{eqnarray}
\lefteqn{\frac{1}{h}\bE K\left(\frac{\WH{\te}_1-x}{h}\right) =
\frac{1}{h}\bE K\left(\frac{\te_1-x}{h}+\frac{\Delta}{h}\right)}\nonumber\\
&& = \frac{1}{h}\bE K\left(\frac{\te_1-x}{h}\right)\label{lab1}\\
&&+\frac{1}{h}\bE\left(\frac{\Delta}{h}
K'\left(\frac{\te_1-x}{h}\right)\right)\label{lab2}\\
&&+\frac{1}{h}\bE\left(\frac{\Delta^2}{h^2}
K''\left(\frac{\te_1-x}{h}\right)\right)\label{lab3}\\
&&+\ldots+\frac{1}{h}\bE\left(\frac{\Delta^{k-1}}{h^{k-1}}
K^{(k-1)}\left(\frac{\te_1-x}{h}\right)\right)\label{lab4}\\
&&+\frac{1}{h}\bE\left(\frac{\Delta^{k}}{h^{k}}
K^{(k)}\left(\frac{Z-x}{h}\right)\right).\label{lab5}\\
\nonumber
\end{eqnarray}
Note that, by the usual properties of a kernel of order $\lfloor\be\rfloor$, see e.g., \cite[Theorem 1.1]{tsybook}, 
for some positive constant $C$ it holds
$$|\eqref{lab1}-\varphi(x)| \leq  C h^{\beta}. $$
It is assumed that the $\WH{\te_j}$'s satisfy \eqref{eq:espcond}, thus
\begin{eqnarray*}
\eqref{lab2} & = & \frac{1}{h}\bE\left(\bE(\Delta|\te_1)\frac{1}{h}
K'\left(\frac{\te_1-x}{h}\right)\right)\\
\\
|\eqref{lab2}| & \leq & \frac{C\log{n}}{nh}
\int\frac{1}{h}\left|K'\left(\frac{u-x}{h}\right)\right|\varphi(u)du\\
 & \leq & \frac{C\log{n}}{nh}\int\left|K'\left(v\right)\right|\varphi(x+vh)dv
\leq \frac{C\log{n}}{nh}\|\varphi\|_{\infty}\int|K'|.\\
\end{eqnarray*}
Splitting $\eqref{lab3}$ using $\cA_1$ and its complement,
$$|\eqref{lab3}|\leq \frac{C\log{n}}{nh^2} \|\varphi\|_{\infty}\int|K''|.$$
By the same argument,
$$|\eqref{lab4}|\leq C\sum_{p=3}^{k-1}\left(\sqrt{\frac{\log{n}}{nh^2}}\right)^p
\leq\frac{C\log{n}}{nh^2}, $$ as soon as $\log{n}/(nh^2)\rightarrow 0$. Finally,
$$|\eqref{lab5}|\leq \frac{C}{h}\left(\sqrt{\frac{\log{n}}{nh^2}}\right)^k
\leq \frac{C}{\sqrt{nh}}\left(\frac{\log{n}}{n^{k-1}h^{2k+1}}\right)^{1/2} $$ Thus
\begin{equation} \label{ebias}
b(x)^2\leq c\left[ h^{2\beta}+\frac{1}{nh}\frac{\log^2{n}}{n h^3}+\frac{1}{nh}
\frac{\log^k{n}}{n^{k-1}h^{2k+1}}\right].
\end{equation}
The variance term is bounded by
\begin{eqnarray*}
v(x) & \leq & \frac{1}{J_nh^2}\bE\left(K\left(\frac{\te_1-x+\Delta}{h}\right)^2\right) \\
 & \leq & \frac{1}{J_nh^2}\bE\left[\bE\left(K\left(\frac{\te_1-x+g(\te_1,\veps)}{h}\right)^2\big|\
\veps\right)\right] \\
& \leq &  \frac{1}{J_nh^2}\bE\left[\int K\left(\frac{u-x+g(u,\veps)}{h}\right)^2
\varphi(u)du\right] \\
& \leq &  \frac{1}{J_nh}\bE\left[\int K\left(v-\frac{g(x+hv,\veps)}{h}\right)^2
\varphi(x+hv)dv\right]\leq \frac{C}{J_nh}\|\varphi\|_{\infty}\|K\|_{\infty}. \\
\end{eqnarray*}
Finally, choosing $k$ large enough in \eqref{ebias}, we obtain, for any $x$ in $\Theta$,
\begin{equation}
  \label{eq:3}
  \bE\left([\hat{\varphi}(x)-\varphi(x)]^2\right) \leq c\left[h^{2\beta}+\frac{1}{nh}\frac{\log^2{n}}{n h^3}\right]+
\frac{C}{J_nh}\|\varphi\|_{\infty}\|K\|_{\infty}
\end{equation}
To obtain the rate of convergence of $\hat{\varphi}$, we distinguish two cases, depending on whether the second or the third term in the
preceding display is dominant.
\begin{itemize}
\item If $J_n \leq (n / \log n)^\frac{2\beta +1}{\beta +2}$, then choosing $h_n = n^\frac{-1}{2\beta+1}$ implies that
$\frac{1}{nh}\frac{\log^2{n}}{n h^3} \leq \frac{C}{J_nh}$, leading to the rate
\begin{equation*}
  \bE\left([\hat{\varphi}(x)-\varphi(x)]^2\right) \leq c J_n^{-\frac{2\beta}{2\beta +1}}.
\end{equation*}
\item If $J_n \geq (n / \log n)^\frac{2\beta +1}{\beta +2}$, then choosing $h_n = (\log n /n)^\frac{1}{\beta+2}$, implies that
$\frac{1}{nh}\frac{\log^2{n}}{n h^3} \geq \frac{C}{J_nh}$, leading to the rate
\begin{equation*}
  \bE\left([\hat{\varphi}(x)-\varphi(x)]^2\right) \leq c n^{-\frac{2\beta}{\beta +2}}.
\end{equation*}
Other choices of $h_n$ can easily be seen to lead to slower rates when optimizing \eqref{eq:3}.
\end{itemize}
\end{proof}

\begin{rem}
Note that the difficulty of the proof relies on the fact that, a priori, $\Delta=\hat{\te}_1-\te_1$ and $\te_1$
{\em are not independent}, see for instance  the expression of the shift estimator given by \eqref{definit}.
Thus one cannot easily change variables in integrals of the type $\int K\left(\frac{u-x+g(u,\veps)}{h}\right)
\varphi(u)du$ since  $g$ depends on $u$.
\end{rem}
\begin{rem} \label{remcond}
 Theorem \ref{nonparam} requires the conditions $\beta>1$ and \eqref{eq:espcond} to be fulfilled. However,
if one (or both) of these two conditions is not assumed, then it is not difficult to check from the preceding
proof, using the rough bound  $|\eqref{lab2}|\leq c/(\sqrt{n} h)  $, that one can still recover a rate of
convergence given by optimization in $h$ of $h^{2\beta}+1/(nh^2)+1/(J_n h)$. This leads to a rate in
$J_n^{-\frac{2\beta}{2\beta +1}}$ (respectively $n^{\frac{-\beta}{2\beta+1}}$), for $J_n$ smaller (resp. larger)
than $n^\frac{2\beta+1}{2\beta+2}$.
\end{rem}

\section{Simulations} \label{ssimul}

In this Section, we first present how the shift estimators studied in Section \ref{siden} can be numerically
implemented. The estimation method proposed here is interesting on its own, since it provides a numerically
tractable semiparametric estimator of the translation parameter and generalizes the penalization method proposed in \cite{LLL04}. Second, we construct the nonparametric estimator of the density defined in Section 3.2 and illustrate its behavior on both simulated data and real data. We point out that in the considered examples, we deal with the case where $f^{[j]}=f$ which is often used in practice where individual effects is only expressed through a warping effect of a main behaviour modeled by a (common) unknown function $f$.

\subsection{Numerical algorithm for shift estimation and extensions}

To compute explicitly $\hat{\te}_n$ given by \eqref{definit} for each curve, one has to choose an appropriate
filter $(h_{k})$. According to Lemma \ref{lem:tsy2}, a good choice of weights should make the remainder term $R_n[h,f^{[j]}]$ small. The authors in \cite{Dalalyan03} determined a sequence $(h_k)$ - roughly, a well-chosen sequence of Pinsker weights - such that the second-order term is optimal from the minimax point of view. However, this choice depends on the regularity parameter of the function $f^{[j]}$, which are not known in practice.\\

Here we use an adaptation of the penalization technique introduced in \cite{LLL04} to determine an appropriate sequence $(h_k)$. However, contrary to that paper, where only projection weights $h_k={\bf 1}_{|k|\leq K}$ are considered, note that the criterion \eqref{definit} enables the use of a much broader variety of weights $(h_k)$, provided these satisfy conditions (C)-(T). As explained below, the use of Pinsker weights, see \eqref{pinsker}, enables a smoothing in the criterion which improves estimation with respect to \cite{LLL04}.\\

We also would like to mention an alternative method based on a data-driven choice of the filter, proposed in \cite{Dalalyan07}. This method is very interesting in particular from the theoretical point of view, since it achieves an optimal minimax second-order term and is adaptive to the regularity of $f^{[j]}$. But the method is asymptotic in nature and, though it performs well for not too complicated signals, our method seems more appropriate for complicated signals (that is, with possibly many non-zero Fourier coefficients) $f^{[j]}$, as in the laser
vibrometry example below.\\


Consider the class of {\em Pinsker-type weights}, depending on the parameters $K$ and $\be$, defined by
\begin{equation} \label{pinsker}
h_{k}=\left[1-\left(k/K\right)^{\be}\right]_{+},\quad k\geq 0,
\end{equation}
and $K$ is called the {\em length} of the sequence of weights.\\
\\
Hence a sequence of weights is characterized by the pair $(\be,K)$. To simplify, we fix the value of $\be$ and
take $\be=3$. Thus the family of weights depends on the single parameter $K$. For any filter sequence of length
$K$, we define
\begin{equation}\label{gdlambda}
\Lambda_{K}(\ta)=\sum_{k=1}^{K}h_{k}\left|\uns\sum_{i=1}^{n}\cos(2\pi k(t_{i}-\ta))Y_{i}\right|^2,
\end{equation}
where $Y_{i},\ i=1,\ldots,n$ is the data corresponding to one curve in model \eqref{eq:pb}. To make the
estimator feasible, we take the values of $\ta$ in a fixed regular grid of mesh $1/m$:
$\{\ta_{1},\ldots,\ta_{i+1}=\ta_{1}+i/m,\ldots,\ta_{max}\}$, of range inferior to $1/2$ (let us recall that the
diameter of $\Theta$ has to be bounded above by $1/2$). Let us define
\begin{eqnarray*}
\WH{\ta}(K)&=&\Argm{\ta_{1},\ldots,\ta_{max}}\Lambda_{K}(\ta),\\
 M(K)&=&\mymax{\ta_{1},\ldots,\ta_{max}}\Lambda_{K}(\ta).
\end{eqnarray*}
{\em Penalization.} We would like to find an adapted sequence of weights, or equivalently an integer $K$. This
is done, as in \cite{LLL04} or \cite{IC05}, using a penalization method. Let
\begin{equation}\label{kchap}
\WH{K}(\al)=\Argm{K_{1},\ldots,K_{max}}{\{-M(K)+\al K\}}.
\end{equation}
The parameter $\alpha$ should yield a trade-off between the fit with the data and the filter length $K$ that can
be viewed as the complexity of the chosen model. We use a data-driven method to find an appropriate $\alpha$.
The idea is to detect the changes in the convex hull of the function $K\rightarrow -M(K)$. Let us recall the
following lemma from \cite{LLL04}.\\
\begin{lem}\label{lemmeKb}
There exist two sequences $K_{1}=1<K_{2}<\cdots$, and $\al_{0}=\pli>\al_{1}>\cdots$, with:
$$\al_{p}=\max_{K_{p}<K\leq K_{max}}\frac{M(K_{p})-M(K)}{K_{p}-K}=\frac{M(K_{p+1})-M(K_{p})}{K_{p+1}-K_{p}},\
p\geq 1,$$ and such that $$\forall\ \al\in(\al_{p};\al_{p+1}],\ \WH{K}(\al)=K_{p}.$$
\end{lem}
Note that connecting the points $(K_{i},M(K_{i}))$ gives the convex hull of the function $K\rightarrow -M(K)$
over the points $K_{1},\ldots,K_{max}$.
Let us define our estimator of the period as
\begin{equation}\label{estnum}
\ta^{*}=\Argm{\ta_{i}}\sum_{p\ :\ \WH{\ta}(K_{p})=\ta_{i}}\{\alpha_{p-1}-\alpha_{p}\}.
\end{equation}
In words, the estimator chooses the point at which the cumulated jump in the derivative is the highest. Note
that different values $\alpha^*$ of $\alpha$ led to such an estimator, which satisfies the identity
$\tau^*=\hat{\tau}(\hat{K}(\alpha^*))$.\\
\\
{\em Illustration of the algorithm.} Figure \ref{plein} illustrates this algorithm with $f$ equal to
$f_1(x)=0.015*\cos\{100\cos(\pi(x-\tau))\}$. The first graph represents the criterion $-M(K)$ together with, in
dotted line, its convex hull. The stem diagram represents the differences $\alpha_{p-1}-\alpha_{p}$
corresponding to the points $K_{p}$. Finally the estimated shifts for the different values of $K$ are
represented by the last graph. The numerical parameters are the following: $n=800$, and the true shift parameter
$\tau=0.35$. The grid for $\tau$ is the regular grid of $[0.25,0.75]$ with $100$ points. Note that, due to the
high level of the noise (small amplitude of the signal with respect to the noise variance), it is difficult to
detect visually the changes in the criterion behavior. Nevertheless the algorithm succeeds in finding the true shift.\\
\\
\indent The number of significative harmonics of $f_1$ is roughly 100, thus if we knew $f_1$, taking $K$ of the
order of 100 would be a reasonable choice in view of \eqref{definit}. In fact, for $k$ much larger than 100, the
corresponding elements in the sum of squares in \eqref{definit} are mainly noise. As we see in Figure
\ref{plein}, with the choice of $\WH{\tau}$ given in \eqref{estnum},
our algorithm chooses $\WH{K}$ in the appropriate interval.\\
\begin{figure}[h!]
\begin{center}
\includegraphics[scale=1]{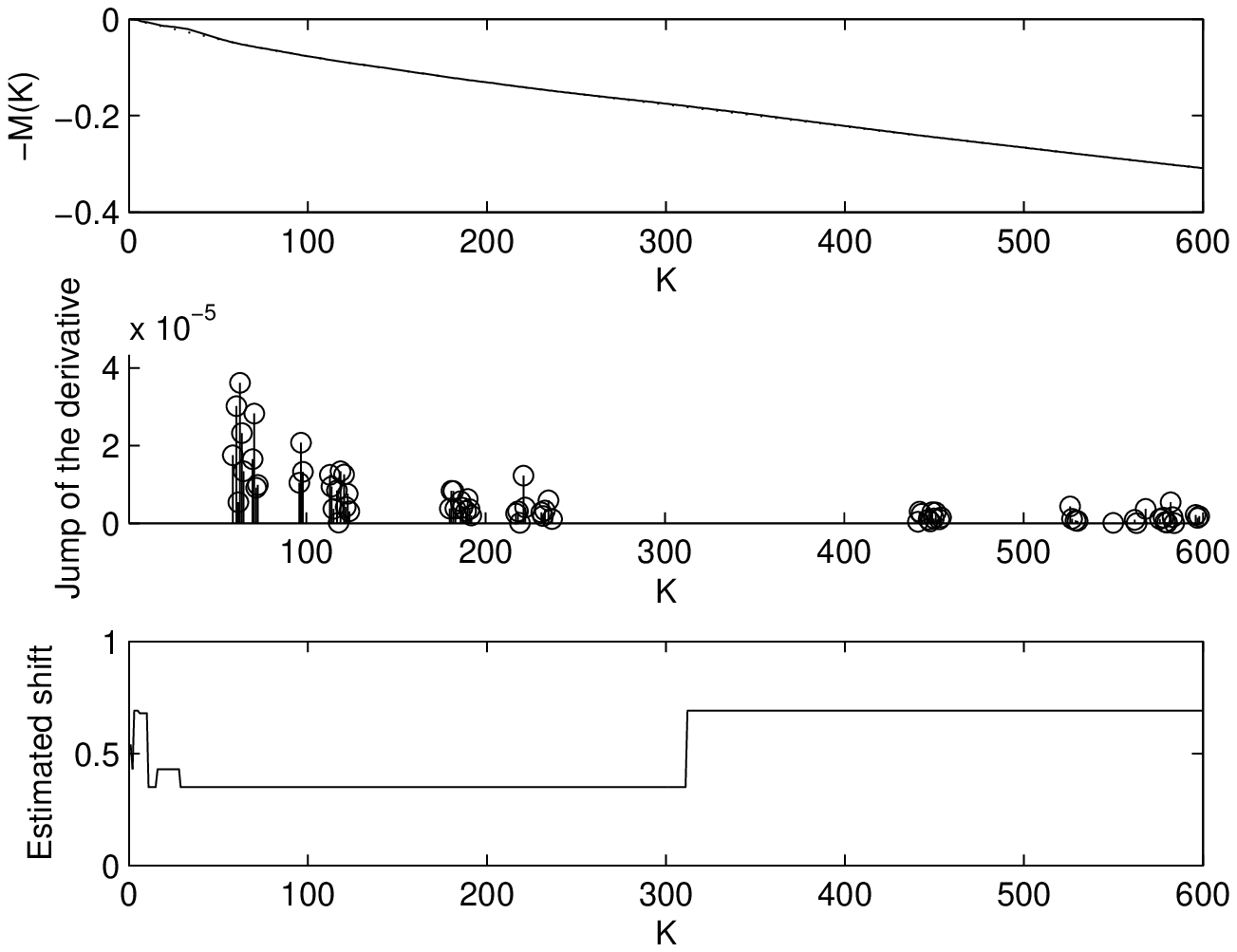}
\caption{Finding the parameter $\WH{\al}$} \label{plein}
\end{center}
\end{figure}

Compared to the method used in \cite{LLL04}, the use of the Pinsker-type weights \eqref{pinsker} allows a
{\em smoothing} with respect to projection weights, making the detection of main jumps in the convex hull of
$K\rightarrow -M(K)$ less sensible to local irregularities of $K\rightarrow M(K)$, which slightly improves estimation. An extensive numerical comparison of the use of the two type of weights, which is beyond the scope of this paper, is carried out in \cite{ICthese} in the case of the period model for discrete design and gains of 10 to 20 \% in the estimation of $\te$ are observed for a laser vibrometry example with the unknown $f$ similar to the function $f_1$ above. \\
\\
{\em The period model.} We note that this algorithm can also be implemented for the period model \eqref{eqperiod} and
more generally if the penalized profile likelihood is known in a closed form. For the period model, the algorithm follows the description above, once one replaces the cosine in \eqref{gdlambda} by its equivalent  $\cos(2\pi k(t_i/\ta))$. A numerical study is carried out in \cite{ICthese}, leading to similar conclusions than the one presented here.

\subsection{Numerical algorithm for density estimation}
Once obtained the estimators of the realizations $\WH{\te}_j,\: j=1,\dots,J$, we can build the estimator of the
density $\vphi$ defined by \eqref{defnp}. We illustrate the good behaviour of our algorithm with three examples. The first one  shows that important features of the target density such as bimodality can be detected with our method. The second example shows that even with quite involved functions, for which the semiparametric step is not easy, the methods performs well, at least if the signal to noise ratio is not too small. The third example deals with a practical application where symmetry can be seen as a sensible assumption. \\
\\
{\em Simulated data (I)}. The function $f$ is the sine function on an half
period, while the law of the shift $\theta$ is a bimodal mixture of compactly supported densities. We perform 50 random translation of
the original curve with $n=100$ observations per curve. In Figure \ref{simule1}, we present the
 observed curves in model \eqref{eq:pb}. To study the performance
 of the estimator described in this paper, we first computed the preliminary estimates $\hat{\theta}_j$
  obtained by the semiparametric method of Section~\ref{siden},
  using the practical algorithm of Section \ref{ssimul}. This set of values
    was then used to build two nonparametric estimators of the density $\varphi$,
    denoted respectively {\it SPGaussian} and {\it SPepanech}, using \eqref{defnp} and
    respectively a Gaussian kernel and Epanechnikov kernel. The smoothing parameters are chosen by
    cross-validation.\\
  We compare their performance with an estimate constructed the following way.
  Using the algorithm described in \cite{Wang99}, applied in the shape invariant model and
  following the lines of Section 3.3 in \cite{Wang99}, we compute nonparametrically the values
  of the warping parameter, used to align the curves to the true shape. Then, using Epanechnikov kernel,
   we build the corresponding density estimator, denoted by {\it NPplug}. \\
\indent Figure 3 carries out the comparison between the preceding estimators.  Visually, the estimators {\it
SPGaussian} and {\it SPepanech} detect the density shape and bimodality and {\it SPepanech}
  matches slightly better the density amplitude. The nonparametric-based kernel estimator {\it NPplug}
  catches the global shape of
   the bimodal density but is too rough, since the method it relies on is far
  too general with respect to the
    semiparametric method designed to handle this particular situation. Hence, plugging a
    semiparametric preliminary estimate into a
     kernel-type estimator leads to a tractable estimator of the density of the shifts without
      knowledge of the shape of the warped function.

\begin{figure}[h!]
\begin{center}
\includegraphics[angle=0,scale=0.7]{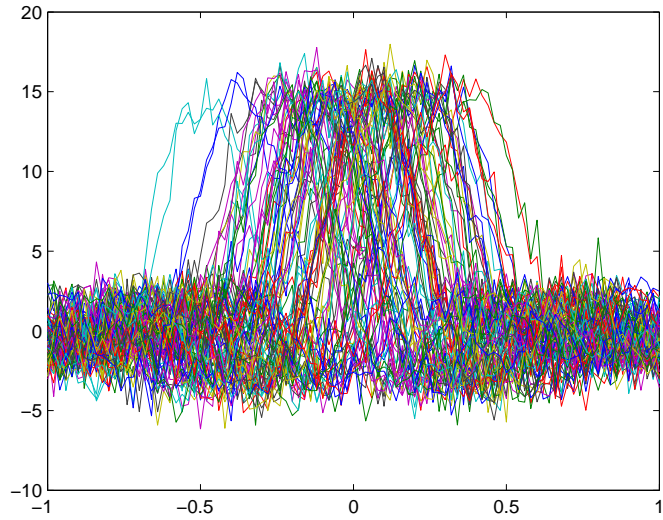}
\caption{Simulated shifted curves.} \label{simule1}
\end{center}
\end{figure}
\begin{figure}[h!]
\begin{center}
\includegraphics[angle=0,scale=0.5]{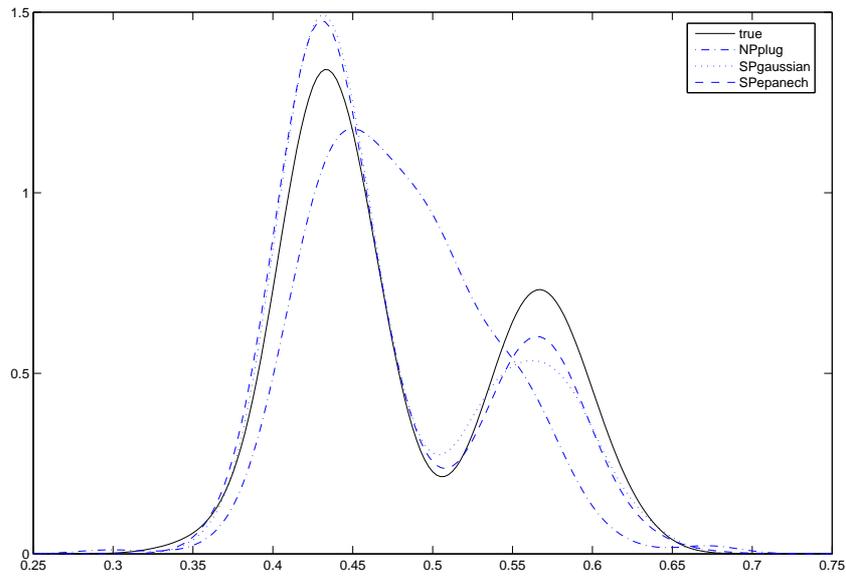}
\caption{Estimators of the shifts density.} \label{simule2}
\end{center}
\end{figure}
\vspace{0.5cm}
\noindent {\em Simulated data (II).}
We consider a function similar to the one introduced in the preceding subsection: $f(x)=\cos\left(100\cos(\pi(x-0.35))\right)$.
Such functions appear for example in laser vibrometry and are
studied in \cite{LLL04} and \cite{IC05} for the period estimation problem. In this case, the semiparametric
estimation step is crucial since the data are fuzzy. In Figure \ref{sim22} we represent the original curve (left picture) and both the  true shift density $\vphi$ 
 and the estimated density, plotted in dotted line (right picture). The curves have been shifted using a compactly supported smooth density $\vphi$, with $n=100$ observations and $J_n=30$. The two functions, the true function $\varphi$  and the estimate, are still visually relatively close.\\
\begin{figure}
\centering
\begin{tabular}{cc}
\epsfig{file=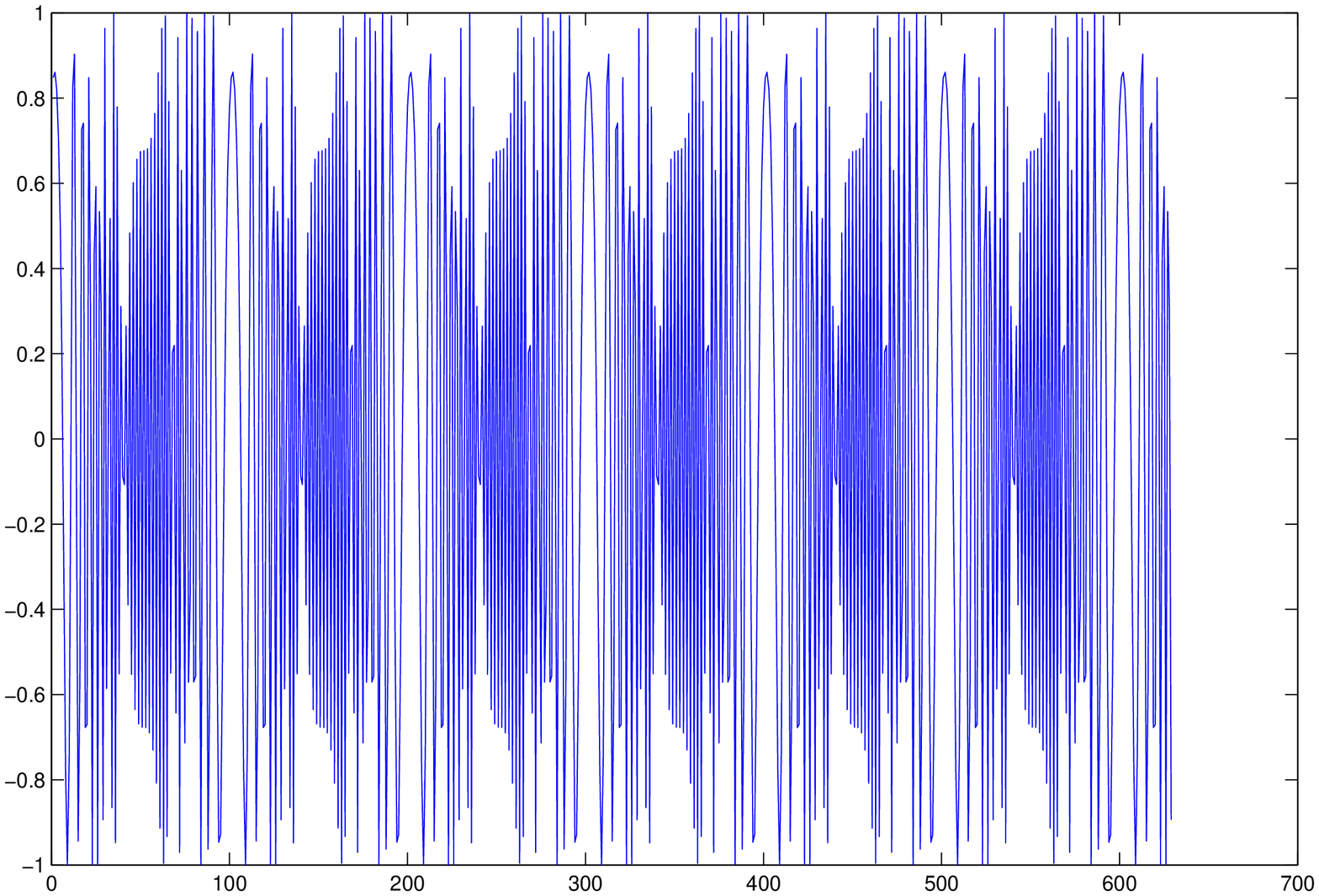,width=0.55\linewidth,clip=} &
\epsfig{file=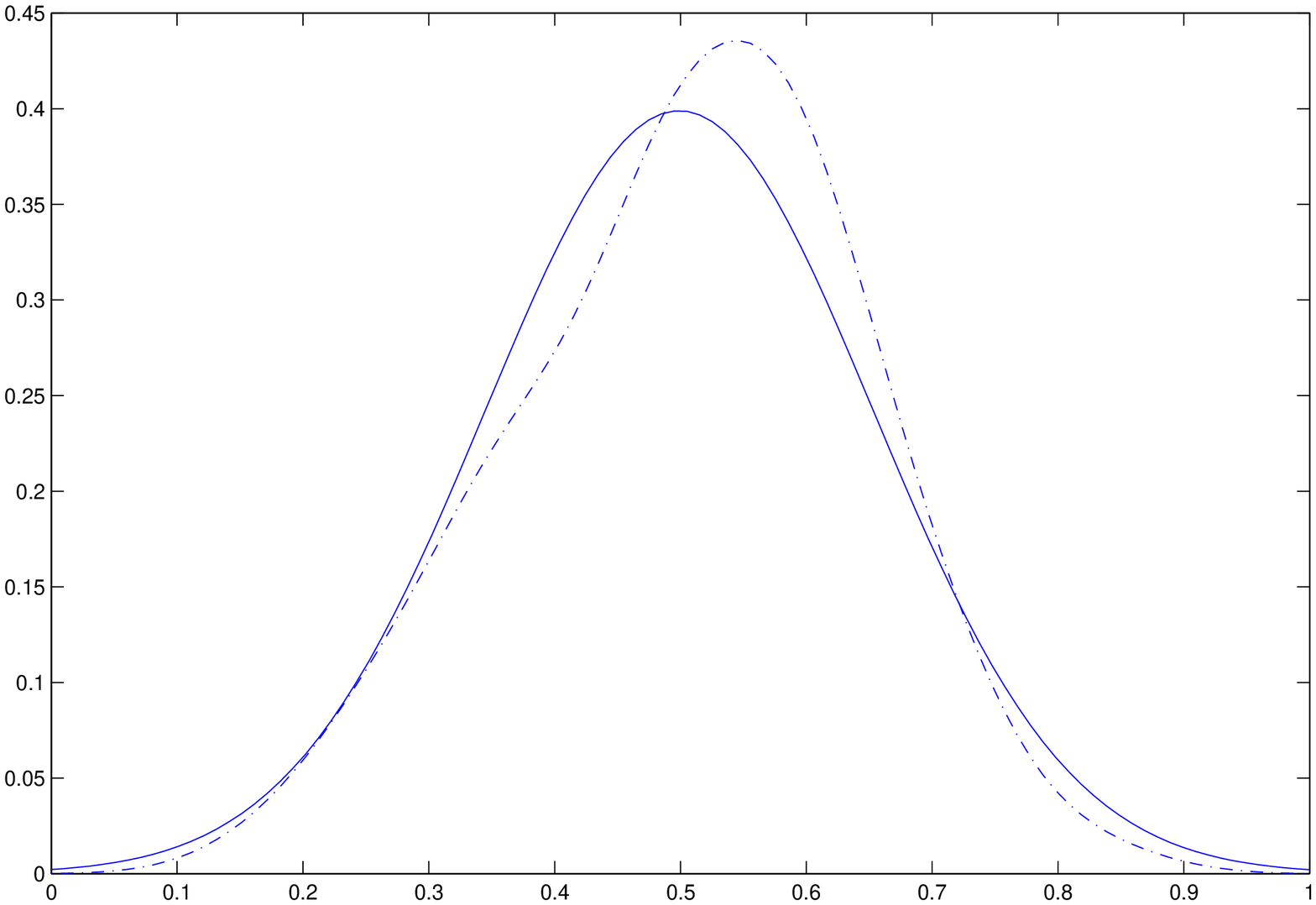,width=0.55\linewidth,clip=}
\end{tabular}
\caption{Simulated laser vibrometry-type functions} \label{sim22}
\end{figure}

\noindent {\em Real data.} We present in Figure~\ref{simule3} an estimation conducted on real data. This data,  provided by ACI-NIM MIST-R (http://www.lsp.ups-tlse.fr/Fp/Loubes/ACI.html), are daily velocities of vehicles on a motorway on the suburbs of Paris.\\
\indent After a preliminary classification which aims at building groups of homogenous curves, we obtain several functional sets, each one representing a particular daily behaviour, as pointed out in \cite{MR2328555}. For one group we get  curves starting and ending at the maximal speed, while presenting some typical patterns which stand for a standard traffic-jam feature, repeated mornings and afternoons. Due to classification, the different features have been split into different classes, as pointed out in \cite{Loubes04}. Hence the curves present some symmetrical aspect but the starting hours of these traffic jams change slightly around a mean time, starting sooner or later each day. Hence, the shift model can be used here, as done also in \cite{Loubes04}.\\
\indent In this study, we get a set of 32 curves with $n=180$ observations which corresponds to a velocity measured every 8 minutes during a day, see the left-hand side of Figure \ref{simule3}. Understanding roadtrafficking behaviour, involves first finding a mean pattern but also studying the density of the random shifts, in order to understand the reasons of this changes around the mean behaviour. The bimodal feature of the estimated density, see the right-hand side of Figure \ref{simule3}, can be later understood as the consequence of different weather conditions on the road network.
\begin{figure}
\begin{tabular}{cc}
\epsfig{file=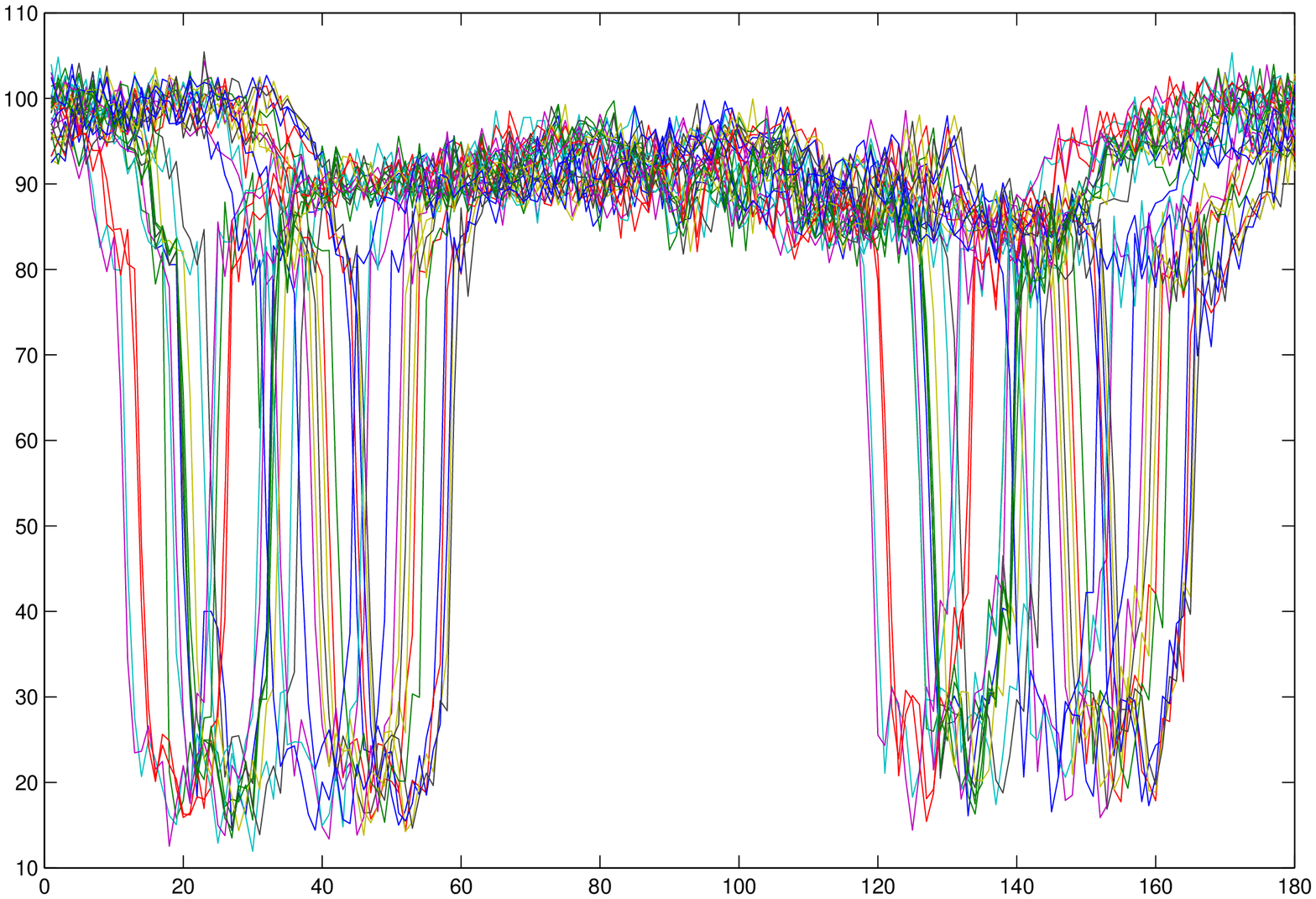,width=0.55\linewidth,clip=} &
\epsfig{file=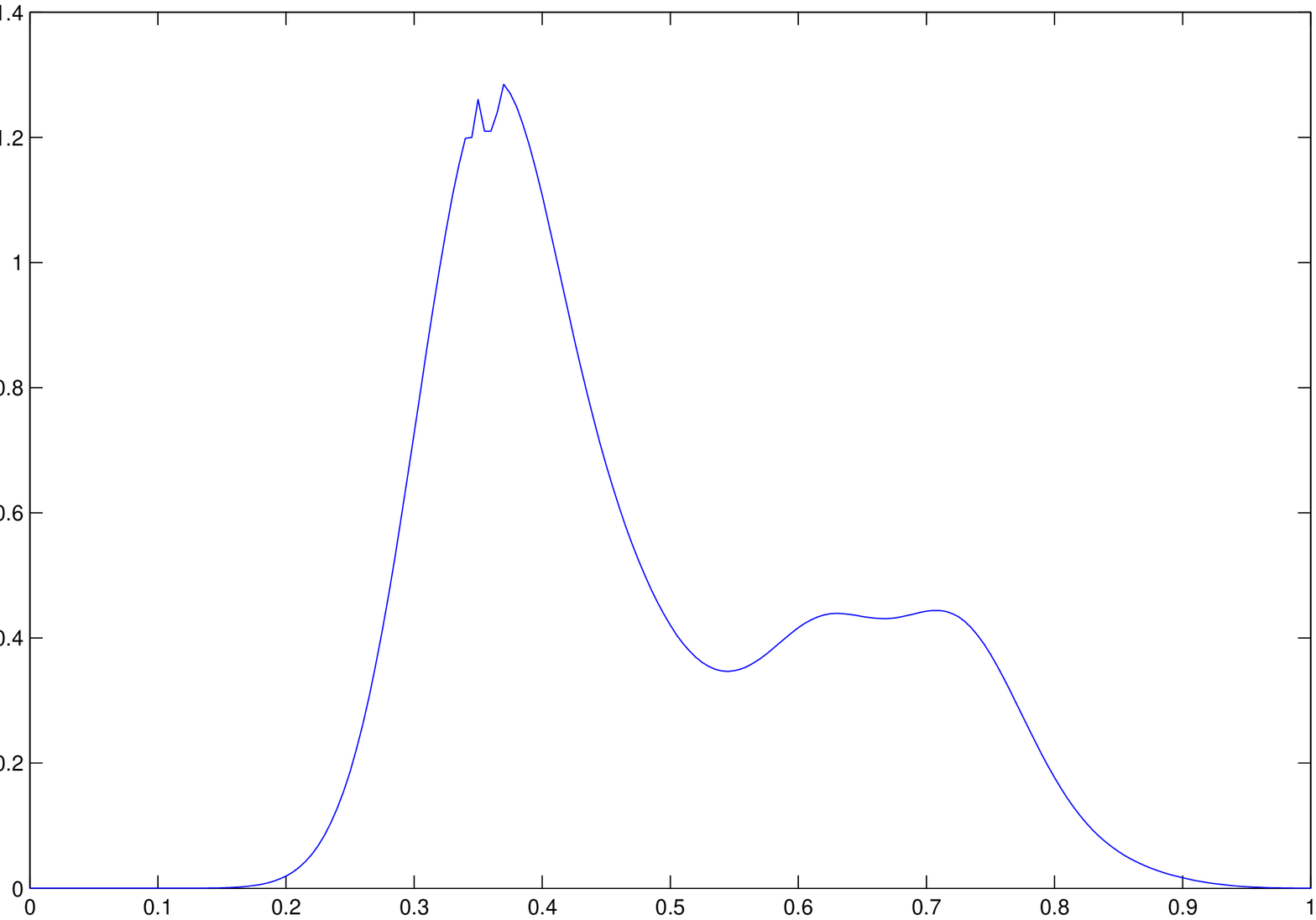,width=0.55\linewidth,clip=}
\end{tabular}
\caption{Real data: velocities curves} \label{simule3}
\end{figure}

\section{Appendix} \label{sappen}

First let us state a useful result about the control of Fourier coefficients by discrete approximations,
which will be used in the sequel to control remainder terms.
Let us denote
\begin{eqnarray}
  \WH{f_{k}} & = & \uns\sum_{i=1}^{n}\rd \cos(2\pi k(t_{i}-\te))f(t_{i}-\te) \label{fk}\\
  \WH{g_{k}} & = & \uns\sum_{i=1}^{n}\rd \sin(2\pi k(t_{i}-\te))f(t_{i}-\te). \label{gk}
\end{eqnarray}
\begin{lem} \label{lemapprox}
For any $f$ in the class $F=F(\rho,C_0)$ satisfying (A1)-(A3), there exists
a constant $C$ depending only on $C_0$ such that, for any $k\geq 1$,
\begin{equation}
 |  \WH{f_{k}} - f_k | \leq C \left(\frac{k}{n}\wedge 1\right) \qquad \text{and} \qquad | \WH{g_{k}} | \leq  C \left(\frac{k}{n}\wedge 1\right), \label{approx}
\end{equation}
where $a\wedge b$ denotes the minimum of the two reals $a$ and $b$.
\end{lem}
\begin{proof}
For any continuously differentiable function $\varphi$ on the interval $[0,1]$ and $t_i=i/n$, it holds
$$ \left| \uns \sum_{i=1}^n \varphi(t_i) - \int_0^1 \varphi(u)du \right| =
\left| \sum_{i=1}^n \int_{t_{i-1}}^{t_{i}} (\varphi(t_i) - \varphi(u)) du \right|.$$
If $\|\varphi'\|_{\infty}$ denotes the supremum norm of the derivative of $\varphi$ on $[0,1]$, we have
\begin{eqnarray*}
\left| \uns \sum_{i=1}^n \varphi(t_i) - \int_0^1 \varphi(u)du \right| & \leq &
 \sum_{i=1}^n \int_{t_{i-1}}^{t_{i}} \|\varphi'\|_{\infty}|t_i-u| du \\
 & \leq & \|\varphi'\|_{\infty}  \sum_{i=1}^n  (t_{i}-t_{i-1})^2/2  = \frac{\|\varphi'\|_{\infty}}{2n}.
\end{eqnarray*}
Now let us apply the preceding to the functions $\varphi_1(u)= \cos(2\pi k (u-\te)) f(u-\te)$ and
$\varphi_2(u)=\sin(2\pi k(u-\te)) f(u-\te)$ respectively. By symmetry and 1-periodicity of $f$, we have
$\int_0^1 \varphi_2(u)du = 0$. For any real $u$,
$$ | \varphi_1'(u) | \leq 2\pi k \| f \|_{\infty}  + \|f'\|_{\infty},$$
and, similarly, the same bound holds for $| \varphi_2'(u) |$. Now observe that $\| f \|_{\infty}$ and
$\|f'\|_{\infty}$ are bounded if $f$ belongs to the class $F$. Indeed, if $f\in F$, then $f'$ is continuously differentiable and 1-periodic thus it is the limit of its Fourier series. For any $u\in[0,1]$, we have
$$ f'(u) = \sum_{k\geq 1} (-2\pi k)f_k \sin(2\pi k u).$$
To see that the latter quantity is bounded it suffices to check that the series $\sum k f_k$ converges. This is a consequence of Cauchy-Schwarz inequality and (A3) since
$$ \sum_{k\geq 1} k|f_k| \leq \left(\sum_{k\geq 1} k^4 f_k^2\right)^{1/2} \left(\sum_{k\geq 1} k^{-2} \right)^{1/2}$$
is bounded. Similarly, $\|f\|_{\infty}$ is bounded by $\sum_{k\geq 1}|f_k|$, which is bounded if
$f$ is in $F$, which implies that $\WH{f_{k}}-f_k$ and $\WH{g_{k}}$ are bounded by a constant times $k/n$. The fact that they are also bounded by a constant follows from the fact that $\|f\|_{\infty}$ is bounded over $F$.
\end{proof}
\begin{lem}\label{lemhk}
Assume that conditions (A1)-(A3) and (C1)-(C3) are fulfilled. Then for some positive constant $C$, denoting
 $ \|h'\|^2= \sum_{k\geq 1} (2 \pi k)^2 h_k^2$,
$$ \sum_{k\geq 1} h_k^2 k^6 \leq C n^2, \quad \sum_{k\geq 1} h_k k^2 |f_k| \leq C\frac{\|h'\|}{\log^2{n}} \quad \text{and} \quad
\sum_{k\geq 1} h_k k^2 |f_k (f_k - \WH{f_k})| \leq C\frac{\|h'\|}{n}. $$
\end{lem}
\begin{proof}[Proof of Lemma \ref{lemhk}]
Note that for any integer $k$, we have $h_k^2 k^6 \leq h_k k^4 \max_{k\geq 1} h_k k^2$. The latter maximum is smaller than the corresponding sum over $k$ which, due to (C3), is at most $D_1 n$. Using (C3) again, one obtains the first inequality. Then due to (C2),
$$ \sum_{k\geq 1}h_k k^2 |f_k| \leq \left(\max_{k\geq 1} h_k k\right) \sum_{k\geq 1} k |f_k| \leq C
\frac{\|h'\|}{\log^2{n}},$$
which yields the second inequality. Finally, using \eqref{approx} and the Cauchy-Schwarz inequality,
$$ \sum_{k\geq 1} h_k k^2 |f_k (f_k - \WH{f_k})| \leq \frac{C}{n} \sum_{k\geq 1} h_k k^3 |f_k|
\leq \frac{C}{n} \left(\sum_{k\geq 1}h_k^2 k^2\right)^{1/2} \left(\sum_{k\geq 1}k^2 f_k^4\right)^{1/2},$$
which yields the third inequality using (A3).
\end{proof}

We can now turn to the proofs of Lemmas \ref{tsy}, \ref{lem:tsy2} and \ref{espcond}.
We shall work conditionally to the event $\te_j=\theta$. To abbreviate the notation, we omit the index $j$
and drop the notation $|\te_j)$. Thus the expectations in the following should be understood at fixed $\te_j$.

The main novelty with respect to \cite{Dalalyan03} consists in proving that the arguments
used by the authors in that paper can be adapted in the discrete setting, by showing that the arguments still go
through when working with the discrete approximations $\WH{f_k}$ and $\WH{g_k}$ instead of $f_k$ and $0$ respectively.

\begin{proof}[Proof of Lemma \ref{tsy}]
The contrast function to maximize is, according to \eqref{definit},
\begin{eqnarray}
L(\tau) & = & \sum_{k\geq 1}h_{k}\left[\uns\sum_{i=1}^{n}\rd \cos(2\pi
k(t_{i}-\tau))Y_{i}\right]^2 \nonumber\\
& = & \sum_{k\geq 1}h_{k}\Bigg[\cos(2\pi k(\te-\tau))\WH{f_{k}}-\sin(2\pi k(\te-\tau))\WH{g_{k}} \nonumber \\
&&  + \frac{1}{\sqrt{n}}(\cos(2\pi k\tau)\xi_{k}+\sin(2\pi k\tau)\xi_{k}^{*})\Bigg]^2, \label{dvtcritere}
\end{eqnarray}
where $\WH{f_{k}}$ and $\WH{g_{k}}$ are defined by \eqref{fk}-\eqref{gk}. The criterion 
$L(\tau)$ is the sum of three terms
$$L(\tau)=\eta_{0}(\tau)+\frac{2}{\sqrt{n}}\|f^{'}\|\eta_{1}(\tau)+\uns \eta_{2}(\tau),$$
where
$$\eta_{0}(\tau)=\sum_{k\geq 1}h_{k}[\cos(2\pi k(\te-\tau))\WH{f_{k}}-\sin(2\pi k(\te-\tau))\WH{g_{k}}]^2$$
$$\eta_{1}(\tau)=\|f^{'}\|^{-1}\sum_{k\geq 1}h_{k}[\cos(2\pi k(\te-\tau))\WH{f_{k}}-\sin(2\pi k(\te-\tau))
\WH{g_{k}}][\cos(2\pi k\tau)\xi_{k}+\sin(2\pi k\tau)\xi_{k}^{*}]$$
$$\eta_{2}(\tau)=\sum_{k\geq 1}h_{k}[\cos(2\pi k\tau)\xi_{k}+\sin(2\pi k\tau)\xi_{k}^{*})]^2.$$
Note that this is the analog of the decomposition of \cite{Dalalyan03}, p. 185, except that here the quantity
$\cos(2\pi k(\te-\tau))\WH{f_{k}} + \sin(2\pi k(\te-\tau))\WH{g_{k}}$ replaces
$\cos(2\pi k(\te-\tau)) f_{k}$. Let us see how the argument is further modified.

The stochastic term $\eta_2$ is exactly the same as in \cite{Dalalyan03}. The term $\eta_1$ is such that its derivative $\eta_1'$  is a zero-mean stationary Gaussian process and one has
\begin{eqnarray*}
 \bE(\eta'_1(\ta)^2) & = & \|f'\|^{-2}\sum_{k\geq 1}h_{k}^2 (2\pi k)^2 (\WH{f}_{k}^2+\WH{g}_{k}^2), \\
 \bE(\eta''_{1}(\ta)^2) & = & 4 \|f'\|^{-2}\sum_{k\geq 1}h_{k}^2 (2\pi k)^4 (\WH{f}_{k}^2+\WH{g}_{k}^2).
\end{eqnarray*}
So, $\eta'_1(\ta)$ has a variance bounded from below by a constant times $\WH{f}_{1}^2$ which is bounded away from zero
for $n$ large enough due to \eqref{approx} and (A3). Moreover, the variance of  $\eta''_{1}$ is
bounded. Hence one can apply Rice formula as in \cite{Dalalyan03} to obtain that there are
 some positive constants $C$ and $D$, such that for all $x>0$,
\begin{equation}\label{eta1}
 \bP(\sup_{\ta\in\Theta} | \eta'_1(\ta) | > x) \leq  C\exp(-Dx^2),
\end{equation}
which is the result obtained in \cite{Dalalyan03}.
Finally we deal with $\eta_{0}$ by writing
\begin{align*}
\eta_{0}(\tau) & = \sum_{k\geq 1}h_{k}\Big[\cos^2(2\pi k(\tau-\te))\WH{f}_{k}^2 - 
2\cos(2\pi k(\tau-\te))\sin(2\pi k(\tau-\te))\WH{f}_{k} \WH{g}_{k} \\
&\qquad +\sin^2(2\pi k(\tau-\te))\WH{g}_{k}^2\Big] = \gamma_{0}(\tau)+\gamma_{1}(\tau)+\gamma_{2}(\tau).
\end{align*}
We have that $\ga_0'(\te)=0$ and $\ga''(\te)\leq -(2\pi)^2\WH{f}_1^2$ is bounded away from zero due to
\eqref{approx} and (A3). Thus, similarly to \cite{Dalalyan03}, one has $\gamma_{0}(\tau)-\gamma_{0}(\te)\leq -C|\tau-\te|^2$ for all $\tau\in\Theta$. Now note that for all real $\ta$,
$$| \ga_1(\ta) + \ga_2(\ta) |  \leq     \sum_{k\geq 1} 2\WH{f}_{k} \WH{g}_{k} + \WH{g}_{k}^2.$$
Hence using \eqref{approx}, the sum $\ga_1(\ta)+\ga_2(\ta)$ is a $O(1/n)$ uniformly in $\ta$. The argument is now completed as follows. Using the obtained bounds, for any positive $x$,
\begin{align*}
&\bP_{\te} \left( |\WH{\theta} - \theta| \sqrt{n \|f^{'}\|^2}> x \right)
\leq \bP_{\te} \left( \sup_{\sqrt{n}\|f^{'}\||\tau-\te|>x}(L(\tau)-L(\te))\geq 0\right)\\
&\leq \bP_{\te} \left(
\sup_{\sqrt{n}\|f^{'}\||\tau-\te|>x}\left[\eta_{0}(\tau)-\eta_{0}(\te)+2\frac{\|f^{'}\|}{\sqrt{n}}
(\eta_{1}(\tau)-\eta_{1}(\te))+\uns(\eta_{2}(\tau)-\eta_{2}(\te))\right]\geq 0\right)\\
&\leq \bP_{\te}\left( \sup_{\sqrt{n}\|f^{'}\||\tau-\te|> x}\eta_{0}(\tau)-\eta_{0}(\te)+
|\tau-\te|\sup_{t\in\Theta}\left\{2\frac{\|f^{'}\||\eta^{'}_{1}(t)|}{\sqrt{n}}+\frac{|\eta^{'}_{2}(t)|}{n}\right\}
\geq 0\right)\\
&\leq \bP_{\te}\left( \sup_{\sqrt{n}\|f^{'}\||\tau-\te|> x} -C + \frac{O(n^{-1})}{|\ta-\te|^2} +
\frac{1}{|\tau-\te|}\sup_{t\in\Theta}\left\{2\frac{\|f^{'}\||\eta^{'}_{1}(t)|}{\sqrt{n}}+\frac{|\eta^{'}_{2}(t)|}{n}\right\} \geq 0\right).
\end{align*}
Now setting $x=x_{n}=K(\log{n})^{1/2}$, for some positive constants $C_1$ and $C_2$ we have
\begin{eqnarray*}
\lefteqn{\bP_{\te}\left( |\WH{\theta} - \theta| \sqrt{n \|f^{'}\|^2}> x_{n}\right)}\\
&& \leq \bP_{\te}\left(  \frac{\sqrt{n}\|f^{'}\|}{x_n} \sup_{t\in\Theta}\left\{2\frac{\|f^{'}\||\eta^{'}_{1}(t)|}{\sqrt{n}}+\frac{|\eta^{'}_{2}(t)|}{n}\right\} \geq - C + \frac{O(n^{-1})}{x_n} \right) \\
&&\leq \bP_{\te}\left(\sup_{t\in\Theta}|\eta'_{1}(t)|\geq Cx_{n}\right)+\bP_{\te}\left(\sup_{t\in\Theta}|\eta'_{2}(t)|\geq C\sqrt{n}x_{n}\right).
\end{eqnarray*}
This is further bounded using \eqref{eta1} for the first term and Lemma 3 in \cite{Dalalyan03} for the second term, which concludes the proof.
\end{proof}
Before we turn to the proof of Lemma \ref{lem:tsy2}, we state a Lemma which summarizes the properties of the criterion $L$, see Equation \eqref{dvtcritere}, and its derivatives. Note that it is in particular a natural adaptation of Lemma 6 in \cite{Dalalyan03}.
\begin{lem}\label{lemel}
Uniformly over $\te\in\Theta$ and $f\in F$, as $n\to\pli$,
\begin{eqnarray}
\bE(L'(\te)^2) & = & \frac{4}{n}\left(\sum_{k\geq 1}(2\pi k)^2 h_k^2 f_k^2 +
(1+o(1))\frac{\|h'\|^2}{n}  \right) \label{lpr}\\
\bE(L''(\te)) & = & -2\sum_{k\geq 1} h_k (2\pi k)^2 f_k^2 + o\left(\frac{\|h'\|^2}{n}\right). \label{lsec}
\end{eqnarray}
Moreover, uniformly over $\te\in\Theta$ and $f\in F$, as $n\to\pli$,
\begin{equation}
\bE\{L''(\te)-\bE(L''(\te))\}^2  =  O(n^{-1}) \qquad \text{and} \qquad
\bE\left(\sup_{\zeta\in\Theta} L^{(3)}(\zeta)^2\right) = O(1).  \label{lder}
\end{equation}
\end{lem}
\begin{proof}[Proof of Lemma \ref{lemel}]
Let us denote
$$ \xi_{k}(\te)=\frac{1}{n}\sum_{i=1}^{n}\sqrt{2}\cos(2\pi k(t_{i}-\te))\veps_{i} \quad \text{and} \quad
\xi_{k}^{*}(\te)=\frac{1}{n}\sum_{i=1}^{n}\sqrt{2}\sin(2\pi k(t_{i}-\te))\veps_{i}. $$
Simple calculations from \eqref{dvtcritere} lead to
\begin{eqnarray}
L'(\te)  & = & 2\sum_{k\geq 1}h_{k}(2\pi k)(\WH{f_{k}}+n^{-1/2}\xi_{k}(\te))
(\WH{g_{k}}+n^{-1/2}\xi_{k}^{*}(\te)) \label{lpthe}\\
L''(\te) & = & 2\sum_{k\geq 1}h_{k}(2\pi k)^2
\left\{-(\WH{f_{k}}+n^{-1/2}\xi_{k}(\te))^2 + (\WH{g_{k}}+n^{-1/2}\xi_{k}^{*}(\te))^2 \right\} \label{lsecth}
\end{eqnarray}
From \eqref{lpthe} we deduce that
\begin{eqnarray*}
\bE(L'(\te)^2) & = & 4\sum_{k\geq 1} h_k^2 (2\pi k)^2 \left(\frac{f_k^2}{n}
 + \frac{1}{n^2}  \right) \\
& & + 4\sum_{k\geq 1} h_k^2 (2\pi k)^2 \left(\WH{f_{k}}^2\WH{g_{k}}^2 + \frac{1}{n}\{2f_k(\WH{f_{k}}-f_k)+(\WH{f_{k}}-f_k)^2\} + \frac{\WH{g_{k}}^2}{n} \right)\\
& & + 8\sum_{k\neq l} h_k h_l (2\pi k)(2\pi l) \WH{f_{k}}\WH{f_{l}}\WH{g_{k}}\WH{g_{l}}.
\end{eqnarray*}
The second term in the last display is bounded using first \eqref{approx} and then Lemma \ref{lemhk}, by  $O(\|h'\|/n^2)$, which is a $o(\|h'\|^2/n^2)$ since due to (C2), the norm $\|h'\|$ tends to $\pli$. 
To bound the third term, note that it is bounded above by the square of
\begin{eqnarray*}
\sum_{k \geq 1} h_k k | \WH{f_{k}}\WH{g_{k}} | & \leq &
\frac{C}{n}\sum_{k \geq 1} h_k \left(k^2 |f_k| + \frac{k^3}{n}\right) 
\leq \frac{C}{n}\left(\frac{\|h'\|}{\log^2 n} + 1\right),
\end{eqnarray*}
using Lemma \ref{lemhk}. The corresponding square is thus a $o(\|h'\|^2/n^2)$. For the second derivative, 
from \eqref{lsecth} we deduce that
\begin{eqnarray}
\bE(L''(\te)) & = & 2\sum_{k\geq 1} h_k(2\pi k)^2(-\WH{f_{k}}^2 + \WH{g_{k}}^2) \label{esplsec}\\
\bE\{L''(\te)-\bE(L''(\te))\}^2 & = & 
\frac{16}{n}\sum_{k\geq 1} h_k^2(2\pi k)^4(\WH{f_{k}}^2 + \WH{g_{k}}^2+\frac{1}{n}) \nonumber
\end{eqnarray}
and proceeding as for \eqref{lpthe}, using \eqref{approx} and Lemma \ref{lemhk}, one obtains \eqref{lsec} and
the first part of \eqref{lder}. Finally, the result about $L^{(3)}$ is obtained as follows. Proceeding as in Lemma 6 in \cite{Dalalyan03}, one easily sees that 
$$ \sup_{\zeta\in\Theta}|L^{(3)}(\zeta)| \leq C\sum_{k\geq 1} h_k k^3 \left(\WH{f_k}^2 + 
\WH{g_k}^2 +\frac{1}{n}(\xi_k^2 + {\xi_k^*}^2)\right). $$
The deterministic part of the last display is bounded by a constant times
$$ \sum_{k\geq 1}  h_k k^3 \left( 2f_k^2 + 2(\WH{f_k}-f_k)^2+\WH{g_k}^2 \right) \leq 
C\sum_{k\geq 1}  h_k \left(k^3 f_k^2 + k^3\frac{k}{n}\right) \leq C'.$$
To obtain the first inequality,  using \eqref{approx} we have bounded one $\WH{f_k}-f_k$ and one $\WH{g_k}$ by $Ck/n$ and the other ones by a constant. The second inequality is obtained using (A3) and (C3), which concludes the proof of the Lemma.
\end{proof}
Note that all the dependence in $\WH{f_{k}}$ and $\WH{g_{k}}$ has vanished in Lemma \ref{lemel}, replaced by results in function of $f_k$ only. In fact, the results of this Lemma are exactly the ones used in \cite{Dalalyan03} to prove the second order expansion, so in fact using this observation there is nothing left to prove to obtain Lemma \ref{lem:tsy2}. We include however the end of the proof for completeness.
\begin{proof}[Proof of Lemma \ref{lem:tsy2}]
As in \cite{Dalalyan03}, the proof is in two steps. The first step is to prove that  $\WH{\ta}$ defined by the following relation has the desired second order expansion,
\begin{equation}\label{tauchapeau}
L'(\te)+(\WH{\ta}-\te)\bE(L''(\te))=0.
\end{equation}
Let us evaluate $\bE((\WH{\ta}-\te)^2I_{n}(f))$, where $I_{n}(f)=n\|f'\|^2$. From \eqref{lpr} and \eqref{lsec} it follows
$$\bE((\WH{\ta}-\te)^2I_{n}(f))=\left[1+\|f'\|^{-2}\sum_{k\geq 1}(h_{k}^2-1)(2\pi k)^2 f_{k}^2
+\|f'\|^{-2}\|h'\|^2(1+o(1))/n\right]$$
$$\times \left[1+\|f'\|^{-2}\sum_{k\geq 1}(h_{k}-1)(2\pi
k)^2f_{k}^2+o(\|h'\|^2/n)\right]^{-2}.$$
Let us expand the square at the denominator and then use a Taylor expansion of the function $x\to (1+x)^{-1}$. Then, one computes the product with the numerator and uses assumption (T). One obtains
$$\bE((\WH{\ta}-\te)^2I_{n}(f))=1+(1+o(1))\|f'\|^{-2}R_{n}(h,f),$$
which is the desired expansion for $\WH{\ta}$.

In a second step, we prove that $\WH{\te}$ and $\WH{\ta}$ are close enough. It is sufficient to do this on the set
$\cA_{1}=\{|\WH{\te}-\te|\leq D(n^{-1}\log{n})^{1/2}\}$ since the probability of its complement is negligible
due to Lemma \ref{tsy}.
By definition of $\hat{\te}$, we have $L'(\hat{\te})=0$. By Taylor's expansion,
\begin{equation*}
0=L'(\hat{\te})=L'(\te)+(\hat{\te}-\te)L''(\te)+\frac{(\hat{\te}-\te)^2}{2}L^{(3)}(\zeta),
\end{equation*}
for some random variable $\zeta$, which can also be written
\begin{eqnarray}
0 & = & L'(\te)+(\hat{\te}-\te)\bE(L''(\te) )\nonumber\\
& & +(\hat{\te}-\te)[L''(\te)-\bE(L''(\te))]+\frac{(\hat{\te}-\te)^2}{2}L^{(3)}(\zeta).\label{contlp}
\end{eqnarray}
Subtracting \eqref{tauchapeau} and \eqref{contlp}, one obtains
\begin{eqnarray*}
\lefteqn{\bE\left((\WH{\te}-\WH{\ta})^2 {\bf 1}_{\cA_{1}}\right)\bE(L''(\te))^2} \\
&& \leq 2\bE\left((\WH{\te}-\te)^2\{L''(\te)-\bE(L''(\te))\}^2
{\bf 1}_{\cA_{1}}\right)+\bE\left((\WH{\te}-\te)^4\sup_{\zeta\in\Theta}|L^{(3)}(\zeta)|^2{\bf 1}_{\cA_{1}}\right).
\end{eqnarray*}
Using \eqref{lder} and the definition of $\cA_1$, one obtains $\bE((\WH{\te}-\WH{\ta})^2I_{n}(f){\bf 1}_{\cA_{1}})\leq Cn^{-1}\log^2{n}$ which is a $o(R_{n}(h,f))$. Finally, by similar arguments, one also sees that $\bE((\WH{\te}-\WH{\ta})(\WH{\ta}-\te){\bf 1}_{\cA_{1}})$ is a
$o(R_{n}(h,f))$ which concludes the proof.
\end{proof}
{\underline{Proof of Lemma \ref{espcond}}:}
\begin{proof}
Starting from \eqref{contlp}, using the triangle and Cauchy-Schwarz inequalities,
\begin{eqnarray*}
\left|\bE\left((\WH{\te}-\te)\right)\bE(L''(\te))\right| & \leq &
|\bE(L'(\te))|\\
&&+(\bE(\WH{\te}-\te)^2)^{1/2}\left\{ \bE[L''(\te)-\bE(L''(\te))]^2\right\}^{1/2}\\
&& +\frac{1}{2}\bE\left\{(\WH{\te}-\te)^2\sup_{\zeta\in\Theta}|L^{(3)}(\zeta)|\right\}.
\end{eqnarray*}
The first term on the right-hand side of the last display can be bounded using \eqref{approx},
$$ \bE(L'(\te)) = \left|2\sum_{k\geq 1}h_{k}(2\pi k)\WH{f_{k}}\WH{g_{k}}\right| 
 \leq C\sum_{k\geq 1}h_{k} \left(k^2\frac{|f_k|}{n} + \frac{k^3}{n^2}\right), $$
 which is a $O(n^{-1})$ due to the assumption on $f$ and (C3). To bound the second term, we use \eqref{lder} and the fact that the result of Lemma \ref{lem:tsy2} implies
$$\bE(\WH{\te}-\te)^2   =   O(n^{-1}). $$
We note that the latter relation could also be checked directly, similarly to the proof of Lemma  \ref{lem:tsy2} but without keeping second order terms.
To bound the third term, we use Cauchy-Schwarz inequality
$$\bE\left\{(\WH{\te}-\te)^2\sup_{\zeta\in\Theta}|L^{(3)}(\zeta)|\right\}
\leq \bE\left((\WH{\te}-\te)^4\right)^{1/2} \bE\left\{\sup_{\zeta\in\Theta}|L^{(3)}(\zeta)|^2\right\}^{1/2}.$$
Now with $\cA_1=\{|\WH{\te}-\te|\leq D(n^{-1}\log{n})^{1/2}\}$, due to Lemma \ref{tsy},
\begin{eqnarray*}
\bE((\WH{\te}-\te)^4) & = & \bE((\WH{\te}-\te)^4{\bf 1}_{\cA_{1}}) + \bE((\WH{\te}-\te)^4{\bf 1}_{\cA_{1}^c}) \\
& \leq & C (\log{n}/n)^2 + C \exp(- d D^2 \log n).
\end{eqnarray*}
Choosing $D$ large enough, we obtain that the third term is a $O(\log{n}/n)$.
The fact that $|\bE(L''(\te))|$ is bounded from above and below by positive constants, which follows from 
\eqref{esplsec} and (A3)-(C3),  yields the announced result.
\end{proof}

\vspace{0.5cm}
\noindent{\bf Acknowledgments.} The authors are grateful to Professor Alexandre Tsybakov for an insightful
remark concerning this work. We also thank a referee whose comments lead us to correct a mistake in a previous version of this work.


\begin{center}
  \hrule
\end{center}

%
%

\end{document}